\documentclass[11pts]{amsart}
\usepackage{amsmath,amssymb,amsfonts,amsthm}
\usepackage{graphicx}
\newtheorem{thm}{Theorem}[section]
\newtheorem{cor}[thm]{Corollary}
\newtheorem{lem}[thm]{Lemma}
\newtheorem{prop}[thm]{Proposition}
\theoremstyle{definition}
\newtheorem{defn}[thm]{Definition}

\theoremstyle{remark}
\newtheorem{rem}[thm]{Remark}
\newtheorem{example}[thm]{Example}
\numberwithin{equation}{section}

\newcommand{\abs}[1]{\left\vert#1\right\vert}
\newcommand{\set}[1]{\left\{#1\right\}}
\newcommand{\R}{\mathbf R}
\newcommand{\eps}{\varepsilon}

\newcommand{\A}{\mathcal{A}}
\newcommand{\F}{\mathbf F}

\newcommand{\End}{\mathrm{End}}
\newcommand{\Z}{\mathbf Z}
\newcommand{\ZZ}{\mathbf Z}
\newcommand{\Q}{\mathbf Q}

\newcommand{\OO}{\mathcal O}
\renewcommand{\O}{\mathcal O}

\newcommand{\id}{\mathit{id}}

\newcommand{\disc}{\mathrm{disc}}

\newcommand{\trace}{\mathrm{Tr}}
\newcommand{\Tr}{\mathrm{Tr}}

\newcommand{\Gal}{\mathrm{Gal}}

\newcommand{\OK}{\mathcal O_K}

\newcommand{\Cl}{\mathrm{Cl}}
\newcommand{\C}{\mathbf C}

\DeclareMathOperator{\Prob}{Prob}
\begin{document}

\title{Isolated Curves for Hyperelliptic Curve Cryptography}%
\author{Wenhan Wang}
\address{Department of Mathematics\\
Box 354350 \\
University of Washington\\
Seattle, WA 98195-4350}%
\email{wangwh@math.washington.edu}%

\thanks{}%

\date{\today}%
\begin{abstract}
We introduce the notion of isolated genus two curves. As there is no known efficient algorithm to explicitly construct isogenies between two genus two curves with large conductor gap, the discrete log problem (DLP) cannot be efficiently carried over from an isolated curve to a large set of isogenous curves. Thus isolated genus two curves might be more secure for DLP based hyperelliptic curve cryptography. We establish results on explicit expressions for the index of an endomorphism ring in the maximal CM order, and give conditions under which the index is a prime number or an almost prime number for three different categories of quartic CM fields. We also derived heuristic asymptotic results on the densities and distributions of isolated genus two curves with CM by any fixed quartic CM field. Computational results, which are also shown for three explicit examples, agree with heuristic prediction with errors within a tolerable range.
\end{abstract}
\maketitle
\section{Introduction and Background}

A genus two hyperelliptic curve defined over a finite field $\F_q$ is a curve that can be written in the Weierstrass form
\begin{equation}
Y^2 + h(X)Y = f(X),
\end{equation}
where $f(x),h(x)\in\F_q[X]$, with $\deg f(X)=5$ or $6$, and $\deg h(x) \leq 2$. In this article we will always assume that all curves considered are non-singular. Note that any genus two curve is always hyperelliptic In many cryptographic aspects, genus two curves behave similarly to elliptic curves, although some differences do exist.

The Jacobian variety of an elliptic curve is naturally isomorphic to the curve itself, i.e., we may define an additive group structure on the curve itself. However, for a genus 2 curve there is no such isomorphism, and we need to work with its Jacobian for cryptography. The Jacobian of a genus two curve $C$ is defined as the quotient of abelian groups $J_C := \mathrm{Div}^0(C)/\mathrm{Prin}(C)$, where $\mathrm{Div}^0(C)$ consists of all degree-zero divisors on $C$, and $\mathrm{Prin}(C)$ is the subgroup of $\mathrm{Div}^0(C)$ consists of all the principal divisors, i.e., divisors generated by a global function on $C$. The Jacobian $J_C$ is also an algebraic variety, and additions on $J_C$ are morphisms of $J_C$ viewed as an algebraic variety.

The security of hyperelliptic curve cryptography on genus 2 curves is based on the discrete log problem (DLP) on a cyclic subgroup of $J_C(\F_q)$. The size of $J_C(\F_p)$ is bounded by the well-known Hasse-Weil bound
\begin{equation}
(\sqrt{q}-1)^4 \leq \abs{J_C(\F_q)} \leq (\sqrt{q}+1)^4.
\end{equation}
To implement a secure cryptosystem, we need a large prime order subgroup of $J_C(\F_q)$ so that $q$ needs to be either a large prime number or a high power of $2$. Since the base field $\F_q$ is very large, there exist a large number of genus 2 curves defined over $\F_q$. There are about $O(q^3)$ non-isomorphic genus two curves defined over $\F_q$. Suppose that the base field is given, we want to find an answer to the question whether all genus 2 curves are equally secure for DLP based cryptography. Before we look for an answer to this question, we may take a look at a similar case for elliptic curves.

Following \cite{Koblitz1}, suppose we are choosing an elliptic curve defined over $\F_p$ to build an elliptic curve cryptosystem. As $p$ is greater than 3, all curves can be written as $y^2 = x^3 + ax + b$, with $a,b\in\F_p$, which gives $O(p^2)$ elliptic curves defined over $\F_p$, but only $O(p)$ non-isomorphic ones. However, different curves are linked by isogenies. By Tate's Theorem, two elliptic curves $E_1$ and $E_2$ are isogenous if and only if they have the same number of points over $\F_p$. Since the length of the Weil interval is $4\sqrt{p}$, there are $O(\sqrt{p})$ isogeny classes defined over $\F_p$. Note that if an isogeny $\phi:E_1\to E_2$ is easy to compute, by which we mean it is less time consuming compared to the known algorithms for solving the discrete log problem, then the discrete log problem on $E_2$ could be reduced to the discrete log problem on $E_1$. This gives us some basic ideas to pick up secure curves for elliptic curve cryptography: (1) We need to know whether it is easy to construct an isogeny from a given curve $E$ to other curves in the same isogeny class. (2) We might want to avoid choosing curves from which isogenies are easy to construct to a large number of other curves, for the reason that if the discrete log problem over any of these curves is solved, then the discrete log problem on $E$ can also be solved, and the cryptosystem is insecure. (3) In that case, our choice of curves should not be random, but quite special.

To answer the first question, in \cite{Koblitz1} the authors introduced the term \emph{conductor gap}. Let us consider an isogeny class of ordinary elliptic curves; it contains an endomorphism class that has CM by the maximal order $\End(E)\cong\OK$ for some elliptic curve $E$ (see \cite{Waterhouse}). To simplify our discussion we make the assumption that $K$ has class number 1, (e.g. $K=\Q[\sqrt{-1}]$). Note that in this case the endomorphism class containing $K$ is actually an isomorphism class, i.e., any curve that has CM by the maximal order $\OK$ is isomorphic to $E$. If $E'$ is another curve, which has CM by an order $\O\subset\OK$, then the complexity to compute the isogeny from $E$ to $E'$ depends on the largest prime divisor of the index $[\OK:\O]$, which is called the \emph{conductor gap}.

For the second question, we might want to choose curves that have CM by the maximal order and have large conductor gaps with any other curves not in the same endomorphism class. Note that the Frobenius endomorphism $\pi$, that sends any point $(x,y)\in E(\F_p)$ to $(x^p,y^p)$, is an endomorphism of $E$. Hence $\End(E)$ contains $\Z[\pi]$, which has index $\sqrt{\frac{\disc(\pi)}{\disc(K)}}$ in $\OK$. Thus if we expect that $E$ has a large conductor gap, it is necessary that there exists a large prime number $l$ that divides $\frac{\disc(\pi)}{\disc(K)}$. Conversely, if there exists a large prime factor $l$ in $\frac{\disc(\pi)}{\disc(K)}$, then we may construct a curve by the CM method \cite{McGuire} that has $l$ as the conductor gap. Hence we reduced the problem to a ring-theoretic problem involving the conductors of all possible orders of $K$.

Here is an example from \cite{Koblitz1}. Suppose we are looking for curves with CM by $K=\Q(\sqrt{-1})$. Choose $B$ to be a random $k$-bit prime, and choose $A$ to be a random even number (perhaps also of $k$ bits, but $A$ may be chosen to have fewer bits) such that (i) $p=A^2+B^2$ is prime, and (ii) either $n=\frac{p+1}{2}-A$ or $n=\frac{p+1}{2}+A$ is a prime. Heuristically one expects to have to test $O(k^2)$ values of $A$ in order to obtain conditions (i) and (ii). Then the curve $E$ over $\F_p$ with equation
\begin{equation}
y^2 = x^3 - ax
\end{equation}
has $2n$ points, where $\alpha\in\F_p$ is a quadratic non-residue whose quartic residue class depends on the sign in $n=\frac{p+1}{2}\mp A$. The trace of $E$ is $\pm 2A$, and its discriminant is $4A^2-4p = -4B^2$. Because $B$ is prime, for $k\geq 80$ it is completely infeasible to transport the discrete log problem on $E$ to that on a generic isogenous curve. Note that $E$ has complex multiplication by the full ring of integers $\Z[\sqrt{-1}]$ (since $\sqrt{-1}$ acts on the curve by $(x,y)\mapsto(-x,iy)$, where $i$ denotes a square root of $-1$ in the finite field); that is, $\End(E)$ has conductor 1. Up to isomorphism $E$ is the only curve in its conductor-gap class, and the endomorphism ring of any of the other isogenous curves has conductor $B$.

The cases are similar in curves with CM by other imaginary quadratic fields. Consider an elliptic curve $E$ with CM by $\OK$, where $K=\Q(\sqrt{-d})$, and $d>0$ is square free. If $\pi\in K$ is a Weil $p$-number, then we have $\pi\bar{\pi}=p$. As $\pi$ can be written as $\pi = A + B\sqrt{-d}$ if $-d\equiv 2,3\mod 4$, and $\pi = A + B\frac{1+\sqrt{-d}}{2}$ if $-d\equiv 1\mod 4$, it follows that
\begin{itemize}
\item If $-d\equiv 1\mod 4$, then $p=A^2+AB+\frac{1+d}{4}B^2$, and the index of $\Z[\pi]$ in $\OK$ is $|B|$. Note that a coarse upper bound for $|B|$ is $2\sqrt{p/d}$
\item If $-d\equiv 2,3\mod 4$, then $p=A^2+dB^2$, and the index of $\Z[\pi]$ in $\OK$ is $|B|$. Note that a coarse upper bound for $|B|$ is $\sqrt{p/d}$.
\end{itemize}

\

\section{Endomorphism Classes of Genus Two Curves over Finite Fields}

Let $k=\F_q$ be the finite field with $q=p^r$ elements. Let $C$ be a non-singular genus two curve defined over $A$, and let $J_C$ be the Jacobian variety of $C$. An endomorphism of $J_C$ is a morphism preserving the abelian group structure. The zero endomorphism is denoted by $0$. Note that the endomorphisms of $J_C$ form a ring, called the ring of endomorphisms of $J_C$, or simply, of $C$. The endomorphism ring of $J_C$ is denoted $\End(J_C)$ or $\End(C)$ if no ambiguity arises.

We say that the curve $C$ has CM by a totally imaginary field $K\subset\C$ if there is an embedding of rings $\iota: K\hookrightarrow\End(C)\otimes\Q$. In our situation, where $C$ is a genus two curve defined over a finite field, there are two cases we need to consider. If the $p$-rank of $C$ is 1 or 2, that is, if $p$ is nonsupsersingular, then $\End(C)\otimes\Q$ is a field of degree 4, and thus the inclusion of fields $\iota$ must be bijective. If the $p$-rank of $C$ is $0$, that is, when $C$ is supersingular, then $\End(C)$ is noncommutative and $\End(C)\otimes\Q$ is a division ring with $K$ contained in its center. In this paper we assume that $C$ is nonsupersingular, that is, $\End(C)\otimes\Q\approx K$ for some quartic CM field $K$. Thus the endomorphism ring $\End(C)$ can be identified with a subring of $K$. Note that $\End(C)$ is a finite rank $\Z$-module, and hence is finitely generated over $\Z$, which shows that $\End(C)$ is integral over $\Z$. Hence $\End(C)$ is contained in the maximal order $\OK$. On the other hand, as the Frobenius endomorphism $\pi\in K$ has minimal polynomial of degree 4, $\Z[\pi,\bar{\pi}]$ is contained in $\End(C)$. This shows that $\End(C)$ has to be an order in $K$, which is summarized as follows.

\begin{prop}
Suppose $C$ is a genus two curve defined over a finite field $\F_p$ with $\End(C)$ identified as a subring in $K$. Then $\Z[\pi,\bar{\pi}]\subseteq\End(C)\subseteq\OK$.
\end{prop}

The image of the Frobenius endomorphism $\pi$ under the map $\iota$ is an algebraic integer $\pi\in K$. The absolute value of all conjugates of $\pi$ is $\sqrt{p}$. An algebraic integer in $K$ whose conjugates have absolute value $\sqrt{p}$ for some prime number $p$ is called a \emph{Weil} $p$-\emph{number} in $K$.

If $\End(C)=\OK$, we say that $C$ has maximal CM by $K$. Two curves $C_1$ and $C_2$ are said to belong to the same endomorphism class if $\End(C_1)=\End(C_2)$. Any isogeny mapping from a curve to another one in the same endomorphism class is called a horizontal isogeny. On the other hand, if $C_1$ is a curve with maximal CM by $\OK$ and $C_2$ has an order $\OO\subseteq\OK$ as its endomorphism ring, then an isogeny $\phi:J_{C_1}\to J_{C_2}$ is called a vertical isogeny. Note that we have the following result on the degree of a vertical isogeny.

\begin{prop}
Let $C_1$ and $C_2$ be as above. Then for any prime number $l$ dividing the index $[\OK:\OO]$, $l$ also divides $\deg \phi$.
\end{prop}

As there is no known efficient algorithm to compute vertical isogenies with large prime degree (greater than 80 bits), it is computationally hard to carry the discrete log problem from $C_1$ to $C_2$. Following the above proposition we shall call the largest prime number that divides $\deg\phi$ the conductor gap between $C_1$ and $C_2$. If a genus two curve $C$ has large conductor gap with any curve $C'$ not in the same endomorphism class, then $C$ is said to be \emph{isolated}. In particular, if the class number of $K$ is 1, then the endomorphism class of $C$ consists of only one isomorphism class. In that case $C$ has large conductor gap with any curve $C'$ not in the same isomorphism class, in which case $C$ is said to be \emph{strictly isolated}. Note that $C$ is strictly isolated if and only if $C$ is isolated and the class number of $K$ equals 1.

It is convenient to consider isolated curves with maximal CM, as there are a number of known algorithms to construct such curves. We first note the following result in \cite{Waterhouse}.

\begin{prop}
For an isogeny class of genus two curves with CM by $K$ and Frobenius endomorphism $\pi\in K$, the possible endomorphism rings are precisely those orders in $\OK$ which contain $\pi$ and $\bar{\pi}$.
\end{prop}

We have immediately the following

\begin{cor}
Suppose $C(\F_p)$ is a genus two curve with CM by $K$ with Frobenius endomorphism $\pi\in K$. If $[\OK:\Z[\pi,\bar{\pi}]]=l$ is a large prime number, then $C$ is isolated.
\end{cor}
\begin{proof}
Since $[\OK:\Z[\pi,\bar{\pi}]]=l$ is a prime number, we know that for any intermediate order $\OO$, we must have $\OO=\OK$ or $\OO = \Z[\pi,\bar{\pi}]$, which shows that $\OK$ and $\Z[\pi,\bar{\pi}]$ are the only two possible endomorphism rings. Hence for any curve $C'$ not in the same endomorphism class, $C'$ must have $\Z[\pi,\bar{\pi}]$ as its endomorphism ring. As the index of $\Z[\pi,\bar{\pi}]$ in $\OK$ is a large prime number, we deduce that $C$ is isolated.
\end{proof}

To determine whether an isolated curve $C$ is strictly isolated, we need information on the size of the endomorphism class containing $C$. The following proposition in \cite{Waterhouse} answers this question.

\begin{prop}
The endomorphism class of genus two curves with CM by an order $\OO\subset K$ contains exactly $\#\Cl(\OO)$ isomorphism classes. In particular, the endomorphism class of genus two curves with maximal CM by $K$ contains exactly $\Cl(K)$ isomorphism classes.
\end{prop}

The above proposition then immediately gives rise to the following

\begin{cor}
Suppose $C(\F_p)$ is a genus two curve with CM by $K$ with $h(K)=1$. If $[\OK:\Z[\pi,\bar{\pi}]]=l$ is a large prime number, then $C$ is strictly isolated.
\end{cor}

\begin{rem}
It is worth noting that other endomorphism classes (other than the one with maximal CM) generally consist of a large number of isomorphism classes, even if $h(K)=1$. We recall the following formula in \cite{Cox} for the ideal class number $\Cl(\OO)$ of a general order
\begin{equation}
\#\Cl(\OO) = \frac{h(K)}{[\OK^{\ast}:\OO^{\ast}]}\frac{\#(\OK/\mathfrak f)^{\ast}}{\#(\OO/\mathfrak f)^{\ast}},
\end{equation}
where $\mathfrak f = \set{x\in\OK|x\OK\subseteq\OO}$ is the largest ideal of $\OK$ which is contained in $\OO$. $\mathfrak f$ is called the conductor of the order $\OO$ in $K$. Roughly speaking, the number of isomorphism classes contained in an endomorphism class is proportional to both $h(K)$ and the index $[\OK:\OO]$. If the index $[\OK:\OO]$ is very large, then there are many isomorphism classes with endomorphism ring $\OO$. Moreover, the endomorphism class with endomorphism ring $\Z[\pi,\bar{\pi}]$ consists of the largest number of isomorphism classes. In other words, smaller endomorphism rings correspond to larger endomorphism classes. If we arrange all the isomorphism classes in an isogeny class according to the partial ordering by inclusion of endomorphism rings (with $\OK$ at the top), then lower endomorphism classes consist of more isomorphism classes. This seems to be a volcano-like structure. The graph whose vertices are formed by isomorphism classes and edges correspond to the isogenies between them is called an isogeny volcano.
\end{rem}

We shall also take into account curves with maximal CM by $K$ where $[\OK:\Z[\pi,\bar{\pi}]]$ is not a prime number but an almost prime number. That is, $[\OK:\Z[\pi,\bar{\pi}]]=lm$ where $l$ is a large prime number and $m$ is a smooth small integer. We assert that in this case any curve $C$ with maximal CM by $K$ has large conductor gap between curves with endomorphism ring $\OO$ such that $l\mid[\OK:\OO]$, or equivalently, $[\OO:\Z[\pi,\bar{\pi}]]\mid m$. Since $m$ is small, there are then only a small number of endomorphism classes that do not have large conductor gap with $C$. Moreover, these endomorphism classes consist of relatively smaller numbers of isomorphism classes. Hence $C$ has large conductor gap with most isomorphism classes in the isogeny class. It is reasonable to call $C$ an \emph{almost isolated} curve.

In the following sections we shall discuss explicitly when $[\OK:\Z[\pi,\bar{\pi}]]$ is a prime number (or an almost prime number) for different kinds of quartic CM fields, and thus give conditions for an endomorphism class with maximal CM by $K$ to be almost isolated, isolated or strictly isolated.

\

\section{Cyclic Extensions with Explicit Integral Basis}

\subsection{The curve $y^2 = x^5 + 1$}
To initiate our discussion of the conductor gap of genus two curves, we would like to first seek a close analogy with the elliptic curve $y^2=x^3+1$, the example given in \cite{Koblitz1}. Note that the elliptic curve $y^2=x^3+1$ has the maximal number of automorphisms, so we shall look for genus two curves with a similar property, i.e., a large automorphism group. The curve $y^2=x^5+1$, $y^2=x^6+1$, $y^2=x^6+x$ are the three genus two curves which look similar to the elliptic curve $y^2=x^3+1$. However, $y^2=x^6+1$ is not simple, which means that its Jacobian is isogenous to a product of two elliptic curves. The curve $y^2=x^6+x$ is isomorphic to the curve $y^2=x^5+1$. Hence we shall first study the curve $C: y^2=x^5+1$.

First notice that if $p$ is congruent to 1 modulo 5, then there exists a primitive fifth root of unity in $\F_p$, which we denote $\overline{\zeta}_5\in\F_p$. Then the map $(x,y)\mapsto (\overline{\zeta}_5x,y)$ induces an automorphism on the Jacobian of $C(\F_p)$. This shows that the endomorphism ring of $J_C$ contains $\Z[\zeta_5]$, where $\zeta_5 = e^{2\pi i/5}$. However, as $K=\Q(\zeta_5)$ is a quartic CM field with $\OK = \Z[\pi,\bar{\pi}]$, we conclude the following

\begin{prop}
If $p\equiv 1\mod{5}$ is a prime number, then the endomorphism ring of $C(\F_p)$ is isomorphic to $\Z[\zeta_5]$. That is, $C$ has maximal CM by $K=\Q(\zeta_5)$. If $p$ is congruent to 2,3, or 4 modulo 5, then $C$ is supersingular.
\end{prop}

Note that the class number $h(K)=1$ for $K=\Q(\zeta_5)$. Hence the endomorphism class containing $C$ consists of only one isomorphism class. Thus $C$ is isolated if and only if it is strictly isolated.

We may pose the following question. For which primes $p\equiv 1\mod{5}$ is the curve $C(\F_p)$ isolated? To answer this question we need to first determine when the index $[\OK:\Z[\pi,\bar{\pi}]]$ is a prime number. Note that this requires the computation of the discriminant of $\pi$.

\begin{lem}
Suppose $\pi$ is a Weil $p$-number in a quartic CM field $K$. If $\bar{\pi}\not\in\Z[\pi]$, then $[\Z[\pi,\bar{\pi}]:\Z[\pi]]=p$.
\end{lem}
\begin{proof}
See \cite{Waterhouse} $\S 6$.
\end{proof}

From the above Lemma we deduce that if $\bar{\pi}\not\in\Z[\pi]$, then
\begin{equation}\label{eq:index}
[\OK:\Z[\pi,\bar{\pi}]] = \sqrt{\frac{\disc(\pi)}{p^2\disc(K)}}.
\end{equation}

In order to compute the discriminant of $\pi$, we need an explicit integral basis for $K=\Q(\zeta_5)$. An obvious integral basis is $\set{1,\zeta_5,\zeta_5^2,\zeta_5^3}$. But computations show that the representation $\pi = A + B\zeta_5 + C\zeta_5^2 + D\zeta_5^3$ is not convenient for studying the primality of $[\OK:\Z[\pi,\bar{\pi}]]$, and it does not easily generalize to a broader classes of curves.

Another way of choosing an integral basis is the following. Let $K_0$ be the maximal real subfield of $K$, which in this case is $K_0=\Q(\sqrt 5)$. Then $\set{1,\frac{1+\sqrt{5}}{2}}$ is an integral basis for $K_0$. We would like to extend this integral basis for $K_0$ to an integral basis for $K$. We shall first take a look at how $\zeta_5$, $\zeta_5^2$, and $\zeta_5^3$ are explicitly written as square roots. Note that we have the following equations.
\begin{eqnarray*}
\zeta_5 &=& -1 + \frac{1+\sqrt{5}}{2} + \frac{1}{4}\left(1-\sqrt{5}+\sqrt{-5-2\sqrt{5}}+\sqrt{-5+2\sqrt{5}}\right) \\
\zeta_5^2 &=& -1 - \frac{1+\sqrt{5}}{2} + \frac{1}{4}\left(1+\sqrt{5}+\sqrt{-5-2\sqrt{5}}-\sqrt{-5+2\sqrt{5}}\right) \\
\zeta_5^3 &=& -1 \qquad\qquad\,\, - \frac{1}{4}\left(1+\sqrt{5}+\sqrt{-5-2\sqrt{5}}-\sqrt{-5+2\sqrt{5}}\right)
\end{eqnarray*}
This shows that the following set
$$\Bigg\{1, \frac{1+\sqrt{5}}{2}, \frac{1}{4}\left(1-\sqrt{5}+\sqrt{-5-2\sqrt{5}}+\sqrt{-5+2\sqrt{5}}\right), $$ $$\frac{1}{4}\left(1+\sqrt{5}+\sqrt{-5-2\sqrt{5}}-\sqrt{-5+2\sqrt{5}}\right)\Bigg\}$$
is an integral basis for $K$. In this integral basis, we may write $\pi$ in the following form
\begin{eqnarray*}
\pi &=& a + b\frac{1+\sqrt{5}}{2} + c\frac{1}{4}\left(1-\sqrt{5}+\sqrt{-5-2\sqrt{5}}+\sqrt{-5+2\sqrt{5}}\right) \\
& & + d\frac{1}{4}\left(1+\sqrt{5}+\sqrt{-5-2\sqrt{5}}-\sqrt{-5+2\sqrt{5}}\right) \\
&=& \frac{1}{4}\left((4a+2b+c+d) + (2b-c+d)\sqrt{5} + (c+d)\sqrt{-5-2\sqrt{5}} + (c-d)\sqrt{-5+2\sqrt{5}}\right)
\end{eqnarray*}
If we set
\begin{eqnarray}
A &=& 4a + 2b + c + d \\
B &=& 2b - c + d \\
C &=& c + d \\
D &=& c - d
\end{eqnarray}
and conversely
\begin{eqnarray}
a &=& \frac{1}{4} A - \frac{1}{4} B - \frac{1}{4} C - \frac{1}{4} D \\
b &=& \frac{1}{2} B + \frac{1}{2} D \\
c &=& \frac{1}{2}C + \frac{1}{2} D \\
d &=& \frac{1}{2} C - \frac{1}{2} D,
\end{eqnarray}
we deduce that
\begin{equation}\label{eq:form-of-pi-zeta-5}
\pi = \frac{1}{4}\left(A + B\sqrt{5} + C\sqrt{-5-2\sqrt{5}} + D\sqrt{-5+2\sqrt{5}}\right).
\end{equation}
Since $\pi$ is a Weil $p$-number, we have $\pi\bar{\pi}=p$. Observe that
\begin{equation}
\pi\bar{\pi} = \frac{1}{16}(A^2 + 5B^2 + 5C^2 + 5D^2) + \frac{1}{16}(2AB + 2C^2 + 2CD - 2D^2)\sqrt{5}.
\end{equation}
Setting $\pi\bar{\pi} = p$ yields
\begin{eqnarray}
\label{eq:conditions-on-A-B-C-D-1}
A^2 + 5B^2 + 5C^2 + 5D^2 &=& 16p, \\
\label{eq:conditions-on-A-B-C-D-2}
AB + C^2 + CD - D^2 &=& 0.
\end{eqnarray}
We have the following result on the discriminant of $\pi$.
\begin{prop}
Let $\pi$ be given in the form (\ref{eq:form-of-pi-zeta-5}), then
\begin{equation}
\disc(\pi) = \frac{125}{16}p^2B^4(C^2-4CD-D^2)^2.
\end{equation}
\end{prop}
\begin{proof}
The proof is by computation. Details can be found in the proof of Proposition \ref{thm:disc-of-pi-cyclic}.
\end{proof}
Since $\disc(K)=125$, we thus deduce from (\ref{eq:index}) that the index of $\Z[\pi,\bar{\pi}]$ in $\OK$ is given by
\begin{equation}
[\OK:\Z[\pi,\bar{\pi}]] = \frac{1}{4}B^2\abs{C^2-4CD-D^2}.
\end{equation}

Now we get back to our question: when is $[\OK:\Z[\pi,\bar{\pi}]]$ a (large) prime number? Note that $B^2$ is a factor of the index, so we take $B^2=1$. However, with $B=\pm1$, (\ref{eq:conditions-on-A-B-C-D-1}) and (\ref{eq:conditions-on-A-B-C-D-2}) become
\begin{eqnarray}
p &=& \frac{1}{16}(A^2 + 5 + 5C^2 + 5D^2) , \\
A &=& \pm(D^2 - CD - C^2).
\end{eqnarray}
Thus we have the following
\begin{prop}
Let $C: y^2=x^5+1$ be defined over $\F_p$, $p\equiv 1\mod{5}$, with Frobenius element $\pi$ given in the form (\ref{eq:form-of-pi-zeta-5}). Then $C$ is (strictly) isolated if
$$p=\frac{1}{16}\left((C^2+CD-D^2)^2 + 5 + 5C^2 + 5D^2\right)$$
and
$$\frac{1}{4}\abs{C^2-4CD-D^2}$$
is a prime number of greater than 80 bits.
\end{prop}
Computations of the frequency of occurrence of isolated curves $C:y^2=x^5+1$ are shown in section 7.

\subsection{Fields similar to $\Q(\zeta_5)$}

The integral basis chosen above for the CM field $K=\Q(\zeta_5)$ in fact easily generalizes to a larger category of CM fields. For this purpose we make the following assumptions on the CM field $K$.
\begin{enumerate}
\item[(i)] $K=\Q(\sqrt{-a-b\sqrt{d}}$ for some positive integers $a,b,d$;
\item[(ii)] $d\equiv 1\mod{4}$ is square-free;
\item[(iii)] $K$ is cyclic, which is equivalent to that $a^2-b^2d=c^2d$ for some $c\in\Z$, $c>0$;
\item[(iv)] The set
$$\Bigg\{ 1, \frac{1+\sqrt{d}}{2}, \frac{1}{4}\left(1+\sqrt{d}+\sqrt{-a-b\sqrt{d}}+\eps\sqrt{-a+b\sqrt{d}}\right),$$
$$\frac{1}{4}\left(1-\sqrt{d}+\sqrt{-a-b\sqrt{d}}-\eps\sqrt{-a+b\sqrt{d}}\right)\Bigg\}$$
is an integral basis for $K$ for $\eps=1$ or $-1$.
\end{enumerate}
We observe that under the above assumptions for $K$, any Weil $p$-number $\pi\in K$ can be written in the form
\begin{equation}\label{eq:form-of-pi-cyclic-integral}
\pi = \frac{1}{4}\left(A + B\sqrt{d} + C\sqrt{-a-b\sqrt{d}} + D\sqrt{-a+b\sqrt{d}}\right).
\end{equation}
Next we shall compute the index of $\Z[\pi,\bar{\pi}]$ in $\OK$. Note that we have the following result on the discriminant of $K$.
\begin{lem}
Let $K$ satisfy the assumptions (i)-(iv) listed above. Then $\disc(K) = a^2d$.
\end{lem}
\begin{proof}
Proof is computational.
\end{proof}
We also have the following computation for the discriminant of $\pi$.
\begin{prop}
Let $\pi$ be a Weil $p$-number in a field $K$ satisfying assumptions (i)-(iv). Then
\begin{equation}
\disc(\pi) = \frac{a^2d}{16}B^4(cC^2-2bCD-cD^2)^2.
\end{equation}
\end{prop}
\begin{proof}
The proof is computational.
\end{proof}
Therefore the index of $\Z[\pi,\bar{\pi}]$ in $\OK$ can be written as
\begin{equation}
[\OK:\Z[\pi,\bar{\pi}]] = \frac{1}{4}B^2\abs{cC^2-2bCD-cD^2}.
\end{equation}
We then obtain the following similar results describing whether a nonsupersingular genus two curve with CM by $K$ is isolated.
\begin{prop}
Let $C$ be a nonsupersingular genus two curve defined over $\F_p$ with Frobenius element $\pi$ given in the form (\ref{eq:form-of-pi-cyclic-integral}). Then $C$ is isolated if there are integers $C,D\in\ZZ$ such that
$$p = \frac{1}{16}\left(\left(\frac{b}{2}C^2+cCD-\frac{b}{2}D^2\right)^2 + d + aC^2 + aD^2\right)$$
and
$$\frac{1}{4}\abs{cC^2-2bCD-cD^2}$$
is a prime number of greater than 80 bits.
\end{prop}

\begin{example}
The above proposition provides an algorithm to find explicitly isolated curves. Here is an explicit example. Let $p = $ $$771091319962693236371145032994729162932757389399122231169290825163207497
497840084770171.$$ Then $y^2=x^5+1$ defined over $\F_p$ is strictly isolated ($h_{\Q(\zeta_5)}=1)$, with a conductor gap $$2955859292970642142002483626678135540313500021819
,$$ a 162-bit prime.
\end{example}

\section{A Broader Class of Isolated Curves}

The isolated curves considered in the previous sections were mainly curves with maximum numbers of symmetry, for example, the curve $y^2=x^5+1$. Curves with less symmetry should also be considered, although the relevant computation in the CM field might be harder. Moreover, curves with index equal to a large prime number occur infrequently among all genus 2 curves. Curves with almost-prime index should also be considered. We should also consider curves with index $[\OK:\Z[\pi,\bar{\pi}]]=lm$, where $l$ is a large prime number, and $m$ is smooth and small. These curves constitute a broader class of curves for use in hyperelliptic curve cryptography.

In this section we shall consider curves with almost-prime index. For this reason it is not necessary here to restrict the quartic fields to the category given in the above section. We consider curves with CM by cyclic quartic imaginary extensions $K/\Q$, with $K = \Q(\sqrt{-a-b\sqrt{d}})$, $a,b,d\in\Z$ and $a^2=(b^2+c^2)d$ for some positive $c\in\Z$.

For simplicity we assume that the Weil $p$-number $\pi$ in $K$ can be written as
\begin{equation}\label{eq:form-of-pi}
\pi = A + Bw + C\eta + D\eta',
\end{equation}
where $\set{1,w}$ is an integral basis for $K_0$, and $\eta\in K$ is an element such that $\eta^2$ is totally negative in $K_0\subseteq\R$. By $\eta'$ we denote one of the non-complex conjugate Galois conjugates, and let $\sigma$ denote the generator of $\Gal(K/\Q)$ such that $\eta'=\eta^{\sigma}$. We first compute the discriminant of $\Z[\pi,\bar{\pi}]$, which is $\disc(\pi)/p^2$.

\begin{lem}\label{thm:disc-of-pi-cyclic}
Let $\pi$ be a Weil $p$-number in a quartic cyclic extension $K/\Q$. Let $\sigma$ be a generator of $\Gal(K/\Q)$ and $\rho$ denote complex conjugation. Then
\begin{equation}
\disc(\pi) = p^2 [\Tr((\pi-\pi^{\sigma})(\pi-\pi^{\sigma\rho}))]^2\Tr(\pi\pi^{\sigma}(\pi-\pi^{\rho})(\pi^{\sigma}-\pi^{\sigma\rho})).
\end{equation}
\end{lem}
\begin{proof}
By the definition of the discriminant we have
\begin{equation}
\disc(\pi)=\left[(\pi-\pi^{\sigma})(\pi^{\rho}-\pi^{\sigma})(\pi^{\sigma}-\pi^{\rho})(\pi^{\sigma\rho}-\pi)\right]^2\left[(\pi-\pi^{\rho})(\pi^{\sigma}-\pi^{\sigma\rho})\right]^2.
\end{equation}
First note that
\begin{eqnarray*}
& &(\pi-\pi^{\sigma})(\pi^{\rho}-\pi^{\sigma})(\pi^{\sigma}-\pi^{\rho})(\pi^{\sigma\rho}-\pi) \\
&=& (\pi\pi^{\rho}+\pi^{\sigma}\pi^{\sigma\rho}-\pi^{\sigma}\pi^{\rho}-\pi\pi^{\sigma\rho})(\pi\pi^{\rho}+\pi^{\sigma}\pi^{\sigma\rho}-\pi\pi^{\sigma}-\pi^{\rho}\pi^{\sigma\rho}) \\
&=& (2p-\pi^{\sigma}\pi^{\rho}-\pi\pi^{\sigma\rho})(2q-\pi\pi^{\sigma}-\pi^{\rho}\pi^{\sigma\rho}) \\
&=& 4p^2 - 2p(\pi\pi^{\sigma}+\pi^{\sigma}\pi^{\rho}+\pi^{\rho}\pi^{\sigma\rho}+\pi^{\sigma\rho}\pi) + p\left[\pi^2+(\pi^{\sigma})^2+(\pi^{\rho})^2+(\pi^{\sigma\rho})^2\right] \\
&=& p\left[\trace(\pi^2) - \trace(\pi\pi^{\sigma}) -\trace(\pi\pi^{\sigma\rho})+\trace(\pi^{\sigma}\pi^{\sigma\rho})\right] \\
&=& p\trace(\pi^2-\pi\pi^{\sigma}+\pi^{\sigma}\pi^{\sigma\rho}-\pi\pi^{\sigma\rho}) \\
&=& p\trace\left[(\pi-\pi^{\sigma})(\pi-\pi^{\sigma\rho})\right].
\end{eqnarray*}
Next note that
\begin{eqnarray*}
& & (\pi - \pi^{\rho})^2(\pi^{\sigma}-\pi^{\rho})^2 \\
&=& (\pi\pi^{\sigma}-\pi^{\sigma}\pi^{\rho}+\pi^{\rho}\pi^{\sigma\rho}-\pi^{\sigma\rho}\pi)^2 \\
&=& (\pi\pi^{\sigma}+\pi^{\rho}\pi^{\sigma\rho})^2 + (\pi^{\sigma}\pi^{\rho}+\pi^{\sigma\rho}\pi)^2 - 2(\pi\pi^{\sigma}+\pi^{\rho}\pi^{\sigma\rho})(\pi^{\sigma}\pi^{\rho}+\pi^{\sigma\rho}\pi) \\
&=& \Tr((\pi\pi^{\rho})^2) + 4p^2 -2p\Tr(\pi^2) \\
&=& \Tr(\pi\pi^{\sigma}(\pi-\pi^{\rho})(\pi^{\sigma}-\pi^{\sigma\rho})).
\end{eqnarray*}
Multiplying both parts together, we obtain the desired result.
\end{proof}

We have thus the following

\begin{cor}
Let $\pi$ be a Weil $p$-number in a cyclic quartic CM field $K$. Then
\begin{equation}
\disc(\Z[\pi,\bar{\pi}]) = [\Tr((\pi-\pi^{\sigma})(\pi-\pi^{\sigma\rho}))]^2\Tr(\pi\pi^{\sigma}(\pi-\pi^{\rho})(\pi^{\sigma}-\pi^{\sigma\rho})).
\end{equation}
\end{cor}
\begin{proof}
It suffices to note that $\disc(\Z[\pi,\bar{\pi}])=\disc(\pi)/p^2$.
\end{proof}

Note that the index of $\Z[\pi,\bar{\pi}]$ in $\OK$ equals $\sqrt{\disc(\Z[\pi,\bar{\pi}])/\disc(K)}$. Hence if we want the index to have a large prime factor, or to be almost prime, we need $\Tr((\pi-\pi^{\sigma})(\pi-\pi^{\sigma\rho}))$ to be as small as possible, and $\Tr(\pi\pi^{\sigma}(\pi-\pi^{\rho})(\pi^{\sigma}-\pi^{\sigma\rho}))$ to be large. Taking into account the specific form of $\pi$ in (\ref{eq:form-of-pi}), we have the following results from computation.

If $d\equiv 1\mod{4}$, then we take $w=\frac{1+\sqrt{d}}{2}$, so that $\set{1,w}$ is an integral basis for $K_0$. Write $\eta = \sqrt{-a-b\sqrt{d}}$, where $a,b$ can be both integers or half-integers to ensure that $\sqrt{-a\pm b\sqrt{d}}\in \OK$. Then $\sigma$ takes $\sqrt{d}$ to $-\sqrt{d}$ and $\sqrt{-a-b\sqrt{d}}$ to $\sqrt{-a+b\sqrt{d}}$. Hence we have the following expressions for all the four conjugates of $\pi$.
\begin{eqnarray}
\pi &=& A + B\frac{1+\sqrt{d}}{2}+ C\sqrt{-a-b\sqrt{d}} + D\sqrt{-a+b\sqrt{d}} \\
\pi^{\sigma} &=& A - B\frac{1+\sqrt{d}}{2} - D\sqrt{-a-b\sqrt{d}} + C\sqrt{-a+b\sqrt{d}} \\
\pi^{\rho} &=& A + B\frac{1+\sqrt{d}}{2} - C\sqrt{-a-b\sqrt{d}} - D\sqrt{-a+b\sqrt{d}} \\
\pi^{\sigma\rho} &=& A - B\frac{1+\sqrt{d}}{2} + D\sqrt{-a-b\sqrt{d}} - C\sqrt{-a+b\sqrt{d}}
\end{eqnarray}

If $d\equiv 2,3\mod{4}$, then we take $w=\sqrt{d}$, so that $\set{1,w}$ is an integral basis for $K_0$. Write $\eta = \sqrt{-a-b\sqrt{d}}$, where $a,b\in\Z$. We have the following expressions for all the four conjugates of $\pi$.
\begin{eqnarray}
\pi &=& A + B\sqrt{d}+ C\sqrt{-a-b\sqrt{d}} + D\sqrt{-a+b\sqrt{d}} \\
\pi^{\sigma} &=& A - B\sqrt{d} - D\sqrt{-a-b\sqrt{d}} + C\sqrt{-a+b\sqrt{d}} \\
\pi^{\rho} &=& A + B\sqrt{d} - C\sqrt{-a-b\sqrt{d}} - D\sqrt{-a+b\sqrt{d}} \\
\pi^{\sigma\rho} &=& A - B\sqrt{d} + D\sqrt{-a-b\sqrt{d}} - C\sqrt{-a+b\sqrt{d}}
\end{eqnarray}

We obtain the following result on the prime factors of $[\OK:\Z[\pi,\bar{\pi}]]$.

\begin{prop}
Let $\pi$ be a Weil $p$-number in the cyclic quartic CM field $K=\Q(\sqrt{-a-b\sqrt{d}})$ in the form (\ref{eq:form-of-pi-d-1-mod-4}) or (\ref{eq:form-of-pi-d-2-3-mod-4}). Then
\begin{enumerate}
\item[(i)] $B$ divides $[\OK:\Z[\pi,\bar{\pi}]]$;
\item[(ii)] If $l>\disc(K)$ is a prime that divides $cC^2-2bCD-cD^2$, then $l$ divides $[\OK:\Z[\pi,\bar{\pi}]]$.
\end{enumerate}
\end{prop}
\begin{proof}
First we observe that $\Z[\pi,\bar{\pi}]\subseteq\Z[\pi,\bar{\pi},\sqrt{d}]\subseteq\OK$ if $d\equiv 2,3\mod{4}$, and $\Z[\pi,\bar{\pi}]\subseteq\Z[\pi,\bar{\pi},\frac{1+\sqrt{d}}{2}]\subseteq\OK$ if $d\equiv 1\mod{4}$. Note that in both cases we have $B\mid \Z[\pi,\bar{\pi}]\subseteq\Z[\pi,\bar{\pi},\sqrt{d}]$ and $B\mid \Z[\pi,\bar{\pi}]\subseteq\Z[\pi,\bar{\pi},\frac{1+\sqrt{d}}{2}]$ respectively. Hence $B$ divides $[\OK:\Z[\pi,\bar{\pi}]]$.

Next, for part (ii), we compute the factor $(\pi-\pi^{\rho})(\pi^{\sigma}-\pi^{\sigma\rho})$ of $\disc(\Z[\pi,\bar{\pi}])$. From the formulas for the conjugates of $\pi$ we derive that
\begin{eqnarray*}
\pi-\pi^{\rho} &=& 2C\sqrt{-a-b\sqrt{d}} + 2D\sqrt{-a+b\sqrt{d}} \\
\pi^{\sigma}-\pi^{\sigma\rho} &=& -2D\sqrt{-a-b\sqrt{d}} + 2C\sqrt{-a+b\sqrt{d}}
\end{eqnarray*}
Hence
\begin{eqnarray*}
(\pi-\pi^{\rho})(\pi^{\sigma}-\pi^{\sigma\rho}) &=& \left(2C\sqrt{-a-b\sqrt{d}} + 2D\sqrt{-a+b\sqrt{d}}\right)\left(-2D\sqrt{-a-b\sqrt{d}} + 2C\sqrt{-a+b\sqrt{d}}\right) \\
&=& 4\sqrt{d}(cC^2-2bCD-cD^2).
\end{eqnarray*}
If $l$ is a prime that divides $cC^2-2bCD-cD^2$ but does not divide $\disc(K)$, then $l$ must be a factor of $\disc(\Z[\pi,\bar{\pi}])/\disc(K)$ and thus $l$ divides $[\OK:\Z[\pi,\bar{\pi}]]$. This completes the proof of (ii).
\end{proof}

We then have the following result from the above proposition on whether a curve with maximal CM by a cyclic quartic field $K$ is almost isolated.

\begin{prop}
Let $C$ be a genus two curve defined over $\F_p$ with maximal CM by a cyclic quartic field $K$. Assume that the Weil $p$-numbers in $K$ is of the form (\ref{eq:form-of-pi-d-1-mod-4}) or (\ref{eq:form-of-pi-d-2-3-mod-4}), depending on whether $d\equiv 1\mod{4}$ or $d\equiv 2,3\mod{4}$. Suppose $B=\pm 1$ and $cC^2-2bCD-cD^2$ is a large prime. Then $C$ is almost isolated.
\end{prop}
\begin{proof}
Note that the discriminant of the basis $\set{1,\sqrt{d},\sqrt{-a-b\sqrt{d}},\sqrt{-a+b\sqrt{d}}}$ in $\OK$ is at most $2^{12}$, which is a smooth and small (compared to $2^{80}$) number. Hence by assumption $[\OK:\Z[\pi,\bar{\pi}]]$ has no large prime factor other than $l$. Hence $C$ is almost isolated.
\end{proof}

\

\section{Non-normal Extensions}

Let $K/\Q$ be a non-normal quartic imaginary extension. Let $\rho$ denote the complex conjugation. Then $K=K^{\rho}$, so the normal closure of $K$, say $L$, has degree 8 over $\Q$, and $\Gal(L/\Q)\approx D_8$, the dihedral group of order 8.  Basic finite group theory tells us that the center of the order 8 dihedral group is isomorphic to $\Z/2\Z$, and there are 4 other non-normal subgroups of order 2, forming 2 pairs, each with 2 conjugate subgroups.

Since $K$ is of degree 4 over $\Q$, $K$ has four embeddings into $\C$. As $K$ is fixed by the complex conjugation $\rho$, we may group these four embeddings into two pairs: $\set{\phi_1,\phi_2}$, $\set{\overline{\phi_1},\overline{\phi_2}}$ such that $\overline{\phi_i}=\phi_i\rho$ for $i=1,2$. We have immediately $\phi_i(K) = \overline{\phi_i}(K)$. In fact there are two ways to form the 4 embeddings into two such pairs. We say that a pair of embeddings $\Phi=\set{\phi_1,\phi_2}$ is a CM type if $\phi_1\neq\phi_2\rho$. Thus a pair of embeddings $\set{\phi_1,\phi_2}$ is a CM type for a non-normal extension $K$ if and only if $\phi_1(K)\neq\phi_2(K)$. We now define the \emph{reflex} of $K$ following \cite{Lang1} and \cite{Shimura}.

\begin{defn}
Let $K$ be a non-normal degree 4 CM field. Then the following two fields are identical; it is called the reflex field of $K$ with respect to the CM type $\Phi=\set{\phi_1,\phi_2}$.
\begin{enumerate}
\item $\Q(\set{\phi_1(x)+\phi_2(x):x\in K})$;
\item $\Q(\set{\phi_1(x)\phi_2(x):x\in K})$.
\end{enumerate}
\end{defn}

As $L$ is dihedral of degree 8 over $\Q$, we may assume that the Galois group $\Gal(L/\Q)$ is generated by $\sigma$ and $\tau$, where $\sigma$ is an element of order 4 and $\tau$ an element of order 2 with $\tau^{-1}\sigma\tau = \sigma^{-1}$. It is possible to choose $\sigma$ such that $\sigma(\phi_1(K))=\phi_2(K)$ for a given CM type $\set{\phi_1,\phi_2}$ of $K$. Let $K_0$ be the maximal real subfield of $K_1=\phi_1(K)$. We assume that $K_0=\Q(\sqrt{d})$ for some square-free positive integer $d$. Note that there exists a non-zero element $x\in K_1$ such that $x^2\in K_0$. This $x^2$ must be totally negative as $K_1$ is a totally imaginary extension. Without loss of generality we may assume $x$ to be an integer in $K_1$ and we may thus write $x = \sqrt{-a-b\sqrt{d}}$ for some $a,b\in\Z$ with $a>0$, and $a^2-b^2d>0$. We may also assume that $\gcd(a,b)$ is square-free and $d\nmid\gcd(a,b)$. Further note that $\set{1,\sqrt{d},\sqrt{-a-b\sqrt{d}},\sqrt{-ad-bd\sqrt{d}}}$ is a $\Q$-basis for $K_1$.

The action of $\sigma$ takes $\sqrt{d}$ to $-\sqrt{d}$ and by our assumption that $\phi_2=\sigma\phi_1$, $\sigma$ maps $\sqrt{-a-b\sqrt{d}}$ to $\sqrt{-a+b\sqrt{d}}$. Hence $\set{1,\sqrt{d},\sqrt{-a+b\sqrt{d}},\sqrt{-ad+bd\sqrt{d}}}$ is a $\Q$-basis for $K_2=\phi_2(K)$.

We now determine the structure of the reflex field $K_1^r$ of $K$ with respect to the CM type $\set{\phi_1,\phi_2}$, say, $K_1^r$. As $\sqrt{-a-b\sqrt{d}}\,\sigma(\sqrt{-a-b\sqrt{d}})=\sqrt{a^2-b^2d}\in K_1^r$, the maximal real subfield $K_0^r$ of $K_1^r$ is $\Q(\sqrt{a^2-b^2d})$. By our assumption $a^2-b^2d$ is positive, square-free. Also note that $\sqrt{a^2-b^2d}$ is not a rational multiple of $\sqrt{d}$, since $\gcd(a,b)$ is not divisible by $d$ by assumption. Observe that
$$\sqrt{-a-b\sqrt{d}}+\sigma(\sqrt{-a-b\sqrt{d}}) = \sqrt{-2a + 2\sqrt{a^2-b^2d}}\in K_1^r.$$
Thus $\set{1,\sqrt{a^2-b^2d},\sqrt{-2a + 2\sqrt{a^2-b^2d}},\sqrt{a^2-b^2d}\sqrt{-2a + 2\sqrt{a^2-b^2d}}}$ is a $\Q$-basis for $K_1^r$.

Since $\set{\phi_1,\phi_2\rho}$ is the other CM-type not conjugate to $\set{\phi_1,\phi_2}$, we obtain another reflex field $K_2^r$ of $K_1$ with respect to $\set{\phi_1,\phi_2\rho}$. Observe that for any $x\in K$, we have $\phi_2\rho(x) = \sigma\rho(\phi_1(x))$. Hence
$$\sqrt{-a-b\sqrt{d}}+\sigma\rho(\sqrt{-a-b\sqrt{d}}) = \sqrt{-2a - 2\sqrt{a^2-b^2d}}\in K_2^r.$$
Thus $\set{1,\sqrt{a^2-b^2d},\sqrt{-2a - 2\sqrt{a^2-b^2d}},\sqrt{a^2-b^2d}\sqrt{-2a - 2\sqrt{a^2-b^2d}}}$ is a $\Q$-basis for $K_2^r$.

Note that $K_1,K_2,K_1^r,K_2^r$ are all non-normal extensions. The only normal quartic subfield of $L$ is the biquadratic real extension $\Q(\sqrt{d},\sqrt{a^2-b^2d})$.

Then we have the following descriptions of the non-trivial subfields of $L$.

\begin{enumerate}
\item $L$ has three quadratic subfields, all of which are real subfields. They are
\begin{itemize}
  \item $K_0 = \Q(\sqrt{d})$;
  \item $K_0^r = \Q(\sqrt{a^2-b^2d})$;
  \item $K_0^{rr} = \Q(\sqrt{a^2d-b^2d^2})$.
\end{itemize}
\item $L$ has five quartic subfields. They are
\begin{itemize}
\item $L_0 = \Q(\sqrt{d},\sqrt{a^2-b^2d})$, which is a normal extension and the maximal real subfield of $L$ and contains all three real quadratic subfields of $L$;
\item $K_1 = \Q(\sqrt{-a-b\sqrt{d}})$;
\item $K_2 = \Q(\sqrt{-a+b\sqrt{d}})$;
\item $K_1^r = \Q(\sqrt{-2a-2\sqrt{a^2-b^2d}})$;
\item $K_2^r = \Q(\sqrt{-2a+2\sqrt{a^2-b^2d}})$.
\end{itemize}
\end{enumerate}

We give an explicit example of these fields. The following example is used in \cite{McGuire} for the construction of $p$-rank 1 genus two curves by the CM method.

\begin{example}
Let $K=Q[X]/(X^4+34X^2+217)$, which is not Galois over $\Q$. Note that explicitly, up to a fixed embedding of $K$ into $\C$, the four roots are $\alpha_1=\sqrt{-17-6\sqrt{2}}$, $\alpha_2=-\sqrt{-17-6\sqrt 2}$, $\beta_1=\sqrt{-17+6\sqrt 2}$, and $\beta_2=-\sqrt{-17+6\sqrt 2}$. Note that $\alpha_1$ and $\alpha_2$, $\beta_1$ and $\beta_2$, respectively are complex conjugates of each other. This field $K$, however, is not normal. To see that $K$ is not Galois, note that the product of $\alpha_1$ and $\beta_1$ is $-\sqrt{217}$, which is a real number that is not in $K_0$, the real subfield of $K$. If we fix a root $\alpha_1$, then the other two embeddings not into $K$ map $\alpha_1$ to $\beta_1$ and $\beta_2$ respectively. Let $\sigma$ denote the map that takes $\alpha_1$ to $\beta_1$, and let $\rho$ denote complex conjugation. In this case, there are two choices of non-conjugate CM types, i.e., $\Phi_1=\{\id, \sigma\}$, and $\Phi_2\{\id, \sigma\rho\}$. The reflex field of $K$ with respect to $\Phi_1$ and $\Phi_2$ are respectively
$$K^r_1=\Q(\sqrt{-17+6\sqrt 2}+\sqrt{-17-6\sqrt{2}}, \sqrt{217}),$$
and
$$K^r_2=\Q(\sqrt{-17+6\sqrt 2}-\sqrt{-17-6\sqrt{2}}, \sqrt{217}).$$
\end{example}

\begin{rem}
In general, suppose we consider an irreducible polynomial $X^4+aX^2+b$ with integer coefficients $a,b\in\Z$ and both roots $r$ and $s$ of $X^2+aX+b$ are totally negative real numbers. Then the field $K=\Q[X]/(X^4+aX^2+b)$ is a quartic CM field with real subfield $K_0=\Q[X]/(X^2+aX+b)=\Q(r)=\Q(s)$. In addition, the four roots of $X^4+aX^2+b$ are obviously $\sqrt r, -\sqrt r, \sqrt s, -\sqrt s$, respectively. Using similar argument to those in the above example, it follows that any one of the reflex field contains $\Q(\sqrt r\sqrt s)=\Q(\sqrt rs)$ as a subfield, i.e., $\Q(\sqrt b)$ is $K^r_0$ if $b$ is not a square.

Also note that $X^4+aX^2+b$ has two ways of factorizations into quadratic polynomials, in $K_1^r$ and $K_2^r$ respectively:
\begin{eqnarray}
X^4+aX^2+b &=& (X^2+\sqrt{2\sqrt b-a}X+\sqrt b)(X^2-\sqrt{2\sqrt b-a}X+\sqrt b) \\
&=& (X^2+\sqrt{-2\sqrt b-a}X-\sqrt b)(X^2-\sqrt{-2\sqrt b-a}X-\sqrt b).
\end{eqnarray}
\end{rem}

We have the following lemma in the computation of the discriminant of a Weil $p$-number in a non-normal quartic extension $K/\Q$.

\begin{lem}
Let $\pi$ be a Weil $p$-number in a quartic cyclic extension $K/\Q$. Let $\sigma$ be a generator of $\Gal(K/\Q)$ and $\rho$ denote complex conjugation. Then
\begin{equation}
\disc(\pi) = p^2 [\Tr_{K/\Q}((\pi-\pi^{\sigma})(\pi-\pi^{\sigma\rho}))]^2\cdot\Tr_{K^r/\Q}(\pi\pi^{\sigma}(\pi-\pi^{\rho})(\pi^{\sigma}-\pi^{\sigma\rho})).
\end{equation}
\end{lem}
\begin{proof}
The computational part of the proof is the same as in the proof of Lemma \ref{thm:disc-of-pi-cyclic}. It then suffices to show that (1) $(\pi-\pi^{\sigma})(\pi-\pi^{\sigma\rho})\in K$; and (2) $\pi\pi^{\sigma}(\pi-\pi^{\rho})(\pi^{\sigma}-\pi^{\sigma\rho})\in K^r$.

To see that $(\pi-\pi^{\sigma})(\pi-\pi^{\sigma\rho})\in K$, note that
$$(\pi-\pi^{\sigma})(\pi-\pi^{\sigma\rho}) = \pi^2 - \pi(\pi^{\sigma}+\pi^{\sigma\rho}) + p.$$
Since $\pi^{\sigma}+\pi^{\sigma\rho}\in K_0\subset K$, we immediately have $(\pi-\pi^{\sigma})(\pi-\pi^{\sigma\rho})\in K$.

Also note that by definition of $K^r$, we have both $\pi\pi^{\sigma}\in K^r$ and $(\pi-\pi^{\rho})(\pi^{\sigma}-\pi^{\sigma\rho})\in K^r$, thus also $\pi\pi^{\sigma}(\pi-\pi^{\rho})(\pi^{\sigma}-\pi^{\sigma\rho})\in K^r$.
\end{proof}

To explicitly compute the discriminant of $\pi$ and the index of $\Z[\pi,\bar{\pi}]$ in $\OK$, we need the following expressions for $\pi$. If $d\equiv 1\mod{4}$, then we consider Weil $p$-numbers of the form
\begin{equation}\label{eq:form-of-pi-d-1-mod-4}
\pi = A + B\frac{1+\sqrt{d}}{2} + C\sqrt{-a-b\sqrt{d}} + D\frac{1+\sqrt{d}}{2}\sqrt{-a-b\sqrt{d}}.
\end{equation}
If $d\equiv 2,3\mod{4}$, then we consider Weil $p$-numbers of the form
\begin{equation}\label{eq:form-of-pi-d-2-3-mod-4}
\pi = A + B\sqrt{d} + C\sqrt{-a-b\sqrt{d}} + D\sqrt{d}\sqrt{-a-b\sqrt{d}}.
\end{equation}

We have the following result on the large prime factors of the index of $\Z[\pi,\bar{\pi}]$.

\begin{prop}
Suppose $l$ is an odd prime with $l>\disc(K)$ such that
\begin{enumerate}
\item[(i)] $l\mid C^2+CD+\frac{1-d}{4}D^2$ if $d\equiv 1\mod{4}$; or
\item[(ii)] $l\mid C^2-dD^2$ if $d\equiv 2,3\mod{4}$.
\end{enumerate}
Then $l$ divides $[\OK:\Z[\pi,\bar{\pi}]]$.
\end{prop}
\begin{proof}
We shall first show that in case (i), $C^2+CD+\frac{1-d}{4}D^2$ is a factor of $(\pi-\pi^{\rho})^2(\pi-\pi^{\sigma})^2$, and is thus a factor of $\disc(\Z[\pi,\bar{\pi}])$. Without loss of generality we may assume that the action of $\sigma$ is given by
$$\sigma: \sqrt{d} \mapsto -\sqrt{d},\quad \sqrt{-a-b\sqrt{d}}\mapsto \sqrt{-a+b\sqrt{d}}.$$
Thus if $\pi$ takes the form (\ref{eq:form-of-pi-d-1-mod-4}), then
\begin{eqnarray}
\pi^{\rho} &=& A + B\sqrt{d} - C\sqrt{-a-b\sqrt{d}} - D\frac{1+\sqrt{d}}{2}\sqrt{-a-b\sqrt{d}} \\
\pi^{\sigma} &=& A - B\sqrt{d} + C\sqrt{-a+b\sqrt{d}} + D\frac{1-\sqrt{d}}{2}\sqrt{-a+b\sqrt{d}} \\
\pi^{\sigma\rho} &=& A - B\sqrt{d} - C\sqrt{-a+b\sqrt{d}} - D\frac{1-\sqrt{d}}{2}\sqrt{-a+b\sqrt{d}}
\end{eqnarray}
Hence
\begin{eqnarray*}
(\pi-\pi^{\rho})(\pi^{\sigma}-\pi^{\sigma\rho}) &=& \left(2 C \sqrt{-a-b\sqrt{d}} + 2D \frac{1+\sqrt{d}}{2}\sqrt{-a-b\sqrt{d}}\right)\cdot \\
& & \cdot\left(2 C \sqrt{-a+b\sqrt{d}} + 2D \frac{1-\sqrt{d}}{2}\sqrt{-a+b\sqrt{d}}\right) \\
&=& 4\sqrt{a^2-b^2d}(C^2+CD+\frac{1-d}{4}D^2).
\end{eqnarray*}
Note that $l\mid(C^2+CD+\frac{1-d}{4}D^2)$ and $l\nmid\disc(K)$, then $l$ must be a factor of $\disc(\Z[\pi,\bar{\pi}]/\disc(K)$, which shows that $l\mid[\OK:\Z[\pi,\bar{\pi}]]$.

Next we treat the case (ii). If $\pi$ takes the form (\ref{eq:form-of-pi-d-2-3-mod-4}), then
\begin{eqnarray}
\pi^{\rho} &=& A + B\sqrt{d} - C\sqrt{-a-b\sqrt{d}} - D\sqrt{d}\sqrt{-a-b\sqrt{d}} \\
\pi^{\sigma} &=& A - B\sqrt{d} + C\sqrt{-a+b\sqrt{d}} - D\sqrt{d}\sqrt{-a+b\sqrt{d}} \\
\pi^{\sigma\rho} &=& A - B\sqrt{d} - C\sqrt{-a+b\sqrt{d}} + D\sqrt{d}\sqrt{-a+b\sqrt{d}}.
\end{eqnarray}
These expressions give
\begin{eqnarray*}
\pi - \pi^{\rho} &=& 2 C \sqrt{-a-b\sqrt{d}} + 2D\sqrt{d}\sqrt{-a-b\sqrt{d}} \\
\pi^{\sigma} - \pi^{\sigma\rho} &=& 2 C \sqrt{-a+b\sqrt{d}} - 2D\sqrt{d}\sqrt{-a+b\sqrt{d}}.
\end{eqnarray*}
Therefore
\begin{eqnarray*}
(\pi-\pi^{\rho})(\pi^{\sigma}-\pi^{\sigma\rho}) &=& \left(2 C \sqrt{-a-b\sqrt{d}} + 2D\sqrt{d}\sqrt{-a-b\sqrt{d}}\right) \\
& & \left(2 C \sqrt{-a+b\sqrt{d}} - 2D\sqrt{d}\sqrt{-a+b\sqrt{d}}\right) \\
&=& 4\sqrt{a^2-b^2d}(C^2-dD^2).
\end{eqnarray*}
Thus we finish the proof by a similar argument as in case (i).
\end{proof}

\

\section{Weil $p$-numbers with prime conductor in quartic cyclic extensions}

Let $K/\Q$ be a totally imaginary quartic extension, denote by $K_0$ its real subextension, and let $\pi$ be a Weil $p$-number in $K$, where $p$ is an odd prime. In the previous sections, we have shown that if $I = [\OK:\Z[\pi,\bar{\pi}]]$ is a prime number larger than 80 bits, then the corresponding Jacobian of the genus two curve is isolated. In this section we derive a heuristic prediction for the asymptotic distribution of Weil $p$-numbers such that both $p$ and $I$ are prime numbers. In order to compute the index $I$, we need to assume that $K$ has integral basis of the form shown in section 3. According to \cite{Spearman1, Spearman2, Spearman3}, the following assumptions ensure that $K$ has the desired form of integral basis.
\begin{enumerate}
\item[(i)] $K_0=\Q(\sqrt{d})$, where $d\equiv 5\mod{8}$, and $K=\Q\left(\sqrt{-a-b\sqrt{d}}\right)$ for some $a,b\in\Z$.
\item[(ii)] Any prime number $l\in\Z$ that ramifies in $K/\Q$ also ramifies in $K_0/\Q$.
\item[(iii)] As $\Z$-modules,
\begin{eqnarray*}
 \OK &\approx& \Z\oplus\Z\frac{1}{2}\left(1+\sqrt{d}\right)\oplus\Z\frac{1}{4}\left(1+\sqrt{d}+\sqrt{-a-b\sqrt{d}}+\epsilon\sqrt{-a+b\sqrt{d}}\right) \\
& & \oplus\Z\frac{1}{4}\left(1+\sqrt{d}-\sqrt{-a-b\sqrt{d}}-\epsilon\sqrt{-a+b\sqrt{d}}\right),
\end{eqnarray*}
    with $\epsilon=1$ or $-1$.
\end{enumerate}

Note that by (1), as $K/Q$ is a cyclic extension, we have $a^2=(b^2+c^2)d$ for some positive integer $c$. Since $d$ is square-free, we have $d\mid a$ and we may write $a=a_0d$. For the simplest cases we note the following result.

\begin{prop}\label{thm:condition-on-K}
Suppose $a_0=1$, $b\equiv 2\mod{4}$, $d\equiv 5\mod{8}$, then the above conditions (i)-(iii) are satisfied for $K=\Q\left(\sqrt{-a-b\sqrt{d}}\right)$.
\end{prop}

For a complete proof and detailed discussion, see \cite{Spearman1, Spearman2, Spearman3}.

In the sequel we shall assume that the condition of Proposition \ref{thm:condition-on-K} is satisfied. Thus $a_0=1$ and $d=b^2+c^2$. Since $\pi$ is a Weil $p$-number, by (iii) we may write $\pi$ in the form

\begin{equation}
\pi = \frac{1}{4}\left(A + B\sqrt{d} + C\sqrt{-d-b\sqrt{d}} + D\sqrt{-d+b\sqrt{d}}\right),
\end{equation}
where $A,B,C,D$ are integers satisfying the equations

\begin{eqnarray}
p &=& \frac{1}{16}(A^2 + B^2d + C^2d + D^2d) \\
AB &=& -C^2\frac{b}{2} - CDc + D^2\frac{b}{2} \\
I &=& \frac{1}{4}\abs{B(C^2c -2CDb - D^2c)}
\end{eqnarray}

If we want both $p$ and $I$ to be prime numbers, we need to take $B=\pm1$. The above equations then become

\begin{eqnarray}
\label{eq:p} p &=& \frac{1}{16}\left(\left(C^2\frac{b}{2} + CDc - D^2\frac{b}{2}\right)^2 + C^2d + D^2d + d\right), \\
\label{eq:I} I &=& \frac{1}{4}\abs{C^2c-2CDb - D^2c}.
\end{eqnarray}

\begin{lem}
In order that $p$ and $I$ are both integers, $C$ and $D$ need to be odd integers.
\end{lem}
\begin{proof}
First note that if both $C$ and $D$ are even integers, then by equation (\ref{eq:p}), $16p$ is odd. Hence $p$ is not an integer.

Next we suppose that one of $C$ and $D$ is odd and the other is even. Note that $c$ is an odd integer, as we assumed that $d = b^2 + c^2$ is odd, and $b$ is even. In this case $C^2c-2CDb-D^2c$ is odd, which yields that $I$ is not an integer.
\end{proof}

We next show that if $3$ divides $c$, then $p$ and $I$ cannot both be prime numbers other than 3.

\begin{lem}
If $3\mid c$, then $3\mid pI$.
\end{lem}
\begin{proof}
Note that  If $3\mid b$, then $3\mid p$ and $3\mid I$. If $3\nmid b$, then $d = b^2+c^2\equiv 1\mod{3}$. Hence $p\equiv (C^2-D^2)^2+C^2+D^2+1\mod{3}$ and $I\equiv \abs{bCD}\mod{3}$. One then verifies that for all choices of congruence classes of $C,D\mod{3}$, we always have $3\mid I$ or $3\mid p$. Thus the desired assertion holds.
\end{proof}

As an immediate corollary, we have

\begin{cor}
Let $p\geq 5$ be a prime number, and $K=\sqrt{-d-b\sqrt{d}}$ satisfy conditions (i)-(iii), $d\equiv 13\mod{24}$, $3\nmid b$. Then there exists no Weil $p$-number $\pi\in K$ such that $[\OK:\Z[\pi,\bar{\pi}]]$ is a prime number larger than 3.
\end{cor}
\begin{proof}
Since $d=b^2+c^2$, and $d\equiv 1\mod{3}$, $3\nmid b$, we have $3\mid c$. Hence by the above lemma, $3\mid pI$. As $p\geq 5$ is a prime number, we have $3\mid I$ and therefore $3\mid[\OK:\Z[\pi,\bar{\pi}]]$.
\end{proof}

\begin{rem}
The only field with $\disc K_0\leq 100$ that has the property described in the above corollary is the field $K=\Q\left(\sqrt{-13-2\sqrt{13}}\right)$. The next such field is $K=\Q\left(\sqrt{-109-10\sqrt{109}}\right)$.
\end{rem}

It is convenient to write
$$\eps = \frac{-c+\sqrt{d}}{b}\quad\textrm{and}\quad \eps'= \frac{-c-\sqrt{d}}{b}.$$
Note that $\eps,\eps'\in K_0$, and $\eps\eps'=-1$. We also note that $\eps$ and $\eps'$ are $\Gal(K_0/\Q)$-conjugate. We may define a $K_0$-linear transformation from $K_0[C,D]\to K_0[U,V]$ as
\begin{eqnarray}
\label{eq:transform:1} U &=& C - \eps D \\
\label{eq:transform:2} V &=& C - \eps'D
\end{eqnarray}
Note that this linear transform induces a linear transformation from $\Z/l\Z[C,D]\to\Z/l\Z[U,V]$, where $l$ is an odd prime number, as long as $d$ is a quadratic residue modulo $l$ and $l\nmid b$. This induced transformation is invertible if and only if $l$ does not divide $d$. Based on the this transformation, we deduce the following result on the factorization of $p$ and $I$, regarded as polynomials in $K_0[U,V]$.

\begin{thm}
With $U$ and $V$ defined as above, we have the following factorizations in $K_0$.
\begin{eqnarray}
\label{eq:factorization-of-p} p &=& \frac{1}{16}\left(\frac{b}{2}U^2+\eps\sqrt{d}\right)\left(\frac{b}{2}V^2-\eps'\sqrt{d}\right), \\
\label{eq:factorization-of-I}I &=& \frac{1}{8}b\abs{\eps'(U-\eps V)(U+\eps V)}
\end{eqnarray}
\end{thm}
\begin{proof}
Both identities can be easily verified.
\end{proof}

It is convenient to define the following subsets of odd prime numbers.
\begin{itemize}
\item[(1)] Let $\mathcal P_C$ denote the set of odd prime numbers that split completely in $K/\Q$.
\item[(2)] Let $\mathcal P_S$ denote the set of odd prime numbers that split in $K_0/\Q$, but do not belong to $\mathcal P_C$.
\item[(3)] Let $\mathcal P_I$ denote the set of odd prime numbers that are inert in $K/\Q$.
\item[(4)] Let $\mathcal R$ denote the set of odd primes that ramify in $K/\Q$.
\end{itemize}
Note that by assumptions (i) and (ii), $\mathcal R$ consists of the prime divisors of $d$. We have the following lemma on the prime numbers that divide $p$.

\begin{lem}\label{lem:divisor-of-p}
If $l\mid p$, then $l\in\mathcal P_C\cup\mathcal R$.
\end{lem}
\begin{proof}
Suppose that $l$ divides $p$. We shall prove by contradiction that either $l$ ramifies in $K/\Q$ or $l$ completely splits in $K/\Q$. If $l$ is inert, then $\mathfrak L=l\OK$ is a prime ideal in $K$. From $\mathfrak L\mid p = \pi\bar{\pi}$ we deduce that either $\mathfrak L\mid \pi$ or $\mathfrak L\mid \bar{\pi}$. Without loss of generality we may assume that $L\mid \pi$. Since $\mathfrak L$ is inert, we have $\mathfrak L^{\sigma}=\mathfrak L$ for all $\sigma\in\Gal(K/\Q)$. Hence $\mathfrak L$ divides all the conjugates of $\pi$, which shows that $\mathfrak L\mid \pi-\pi^{\sigma}+\bar{\pi} - \bar{\pi}^{\sigma}=\sqrt{d}$. This yields that $l$ ramifies in $K_0/\Q$, a contradiction.

Next we assume that $l$ splits in $K_0/\Q$ as $l\mathcal O_{K_0}=\mathfrak l_1\mathfrak l_2$, but $\mathfrak l_i$ remains prime in the extension $K/K_0$, for $i=1,2$. Write $\mathfrak L_i = \mathfrak l_i\OK$ for $i=1,2$, which are both prime ideals in $K$. Note that $\mathfrak L_1^{\sigma} = \mathfrak L_2$ and $\mathfrak L_2^{\sigma}=\mathfrak L_1$, where $\sigma$ is a generator of $\Gal(K/\Q)$. From $l\mid p$ we deduce that $\pi\bar{\pi}\in \mathfrak L_1\mathfrak L_2=\mathfrak L_1\cap\mathfrak L_2$. Hence $\pi\bar{\pi}\in\mathfrak L_i$ for both $i=1,2$. Hence for each $i=1,2$, $\mathfrak L_i$ divides one of $\pi$ or $\bar{\pi}$. However, as $\mathfrak L_i$ is invariant under complex conjugation, $\mathfrak L_i$ divides both $\pi$ and $\bar{\pi}$ for both $i=1,2$. Therefore both $\mathfrak L_1$ and $\mathfrak L_2$ divide all four conjugates of $\pi$, and thus so does $\mathfrak L=\mathfrak L_1\cap \mathfrak L_2$. Hence $\mathfrak L\mid\sqrt{d}$, which yields a contradiction.
\end{proof}

\

The next proposition gives the probability that a random prime number $l$ divides $I$. We first prove the following lemma.

\begin{lem}
If $l$ is an odd prime number such that $l\mid bc$ and $l\nmid d$, then $l$ splits in $K_0/\Q$.
\end{lem}
\begin{proof}
First note that $d=b^2+c^2$. If $l\mid b$, then $d\equiv c^2\mod{l}$. If $l\mid c$ then $d\equiv b^2\mod{l}$. Hence in both cases, $d$ is a quadratic residue modulo $l$. Thus the minimal polynomial of $\frac{1+\sqrt{d}}{2}$ factors in $\Z/l\Z$, and $l$ splits in $K_0/\Q$ by Kummer's theorem.
\end{proof}

\

\begin{prop}\label{thm:2:gen}
Let $C$ and $D$ be random odd integers. If $l$ is an odd prime number, then
\begin{equation}
\Prob(l\nmid I) = \left\{
                    \begin{array}{cl}
                      (1-1/l)^2, & \hbox{if $l$ splits in $K_0/\Q$;} \\
                      (1-1/l^2), & \hbox{if $l$ is inert in $K_0/\Q$;} \\
                      (1-1/l), & \hbox{if $l\mid d$.}
                    \end{array}
                  \right.
\end{equation}
\end{prop}
\begin{proof}
First, if $l$ splits in $K_0/\Q$, then $\sqrt{d}\in\Z/l\Z$. First, if we assume that $l\nmid b$, we then have the factorization
\begin{equation*}
I \equiv \frac{1}{8}b\abs{\eps'(U-\eps V)(U+\eps V)} \mod{l}.
\end{equation*}
Thus $l\mid I$ if and only if $l\mid U+\eps V$ or $l\mid U-\eps V$. Since $C\mod{l}$ and $D\mod{l}$ are randomly chosen in $\Z/l\Z$, as $l\neq d$, the two factors $U+\eps V$ and $U-\eps V$ also range randomly in $\Z/l\Z$. Hence the probability that $l$ divides each of the factors is $1/l$. Moreover, since $C\mod{l}$ and $D\mod{l}$ are independent random variables, the probability that $l$ divides both factors is $1/l^2$. If $l\mid b$, then $I\equiv \frac{1}{4}\abs{cC^2-cD^2}\equiv\frac{1}{4}\abs{c(C+D)(C-D)}\mod{l}$. Thus $l$ divides $I$ if and only if $l$ divides $C+D$ or $C-D$. Therefore in all the three cases we have $\Prob(l\nmid I)=1-2/l+1/l^2 = (1-1/l)^2$.

Next we consider the case when $l$ is inert in $K_0/\Q$. Note that in this case $d$ is a non-residue modulo $l$. We may write $I = \frac{1}{4}\abs{(C-\frac{b}{c}D)^2-\frac{a_0^2d}{c^2}D^2}$. Thus $I\equiv 0\mod{l}$ if and only if $(C-\frac{b}{c}D)^2-\frac{a_0^2d}{c^2}D^2\equiv 0\mod{l}$, which is equivalent to $C\equiv D\equiv 0\mod{l}$. This shows that $l$ divides $I$ if and only if $l$ divides both $C$ and $D$, which occurs with probability $1/l^2$. Therefore $\Prob(l\nmid I) = 1-1/l^2$.

If $l\mid d$, then $I\equiv \frac{1}{4}(C-\frac{b}{c}D)^2\mod{l}$. Hence $d\mid I$ if and only if $l\mid C-\frac{b}{c}D$. Since $C-\frac{b}{c}D$ is a linear combination of $C$ and $D$, $l\mid C-\frac{b}{c}D$ with probability $1/d$. Therefore $\Prob(l\nmid I) = 1-1/l$.
\end{proof}

We need the following lemma to prove our main result on the probability that $l$ does not divide either $p$ or $I$.

\begin{lem}\label{lem:residue}
Suppose $l$ is an odd prime number and $l\nmid b$. Then the following are equivalent.
\begin{enumerate}
\item[(i)] $l\in\mathcal P_C$;
\item[(ii)] $-\frac{2\eps\sqrt{d}}{b}$ is a square modulo $l$;
\item[(iii)] $\frac{2\eps'\sqrt{d}}{b}$ is a square modulo $l$.
\end{enumerate}
\end{lem}
\begin{proof}
The equivalence $(ii)\Leftrightarrow(iii)$ is clear, as $\eps\cdot(-\eps')=-1$. It then suffices to show that $(i)\Leftrightarrow(ii)$.

$(i)\Rightarrow(ii)$. Suppose $l\in\mathcal P_C$, then Kummer's theorem yields that $\sqrt{-d\pm b\sqrt{d}}\in\Z/l\Z$. As $\sqrt{-d-b\sqrt{d}}+\sqrt{-d+b\sqrt{d}}=\sqrt{-2d+2c\sqrt{d}}$, we have $-2d+2c\sqrt{d}$ is a square in $\Z/l\Z$. Hence $-\frac{2\eps\sqrt{d}}{b} = -\frac{\sqrt{-2d+2c\sqrt{d}}}{b^2}$ is also a square in $\Z/l\Z$.

$(ii)\Rightarrow(i)$. Suppose $-\frac{2\eps\sqrt{d}}{b}$ is a square in $\Z/l\Z$, then $$\sqrt{-d-b\sqrt{d}}=\frac{1}{2}\left(\sqrt{-2d+2c\sqrt{d}}-\sqrt{-2d-2c\sqrt{d}}\right)\in\Z/l\Z.$$ By Kummer's theorem, $l$ splits completely in $K/\Q$.
\end{proof}

\begin{prop}
Suppose $l$ is a prime number and $l\nmid b$. Then the probability of $l$ not dividing either $p$ or $I$ is
\begin{equation}
\Prob(l\nmid p\textrm{ and }l\nmid I) = \left\{
                                          \begin{array}{cl}
                                            (1-3/l)^2, & \hbox{if $l\in\mathcal P_C$;} \\
                                            (1-1/l)^2, & \hbox{if $l\in\mathcal P_S$;} \\
                                            (1-1/l^2), & \hbox{if $l\in\mathcal P_I$;} \\
                                            1-1/l, & \hbox{if $l\mid d$;} \\
                                            1, & \hbox{if $l=2$.}
                                          \end{array}
                                        \right.
\end{equation}
\end{prop}
\begin{proof}
First we consider the case when $l\in\mathcal P_C$. Note that in this case $\sqrt{d}\in\Z/l\Z$, and hence the transforms (\ref{eq:transform:1}) and (\ref{eq:transform:2}) are well-defined in $\Z/l\Z$. Since $C$ and $D$ are randomly distributed in $\Z/l\Z$, the pair $(U,V)$ also ranges randomly in the affine space $\A^2(\Z/l\Z)$. To compute the probability, note that the condition $l\nmid I$ is equivalent to both $l\nmid U + \eps V$ and $l\nmid U - \eps V$. Hence $l\nmid I$ if and only if the pair $(U,V)\in\A^2(\Z/l\Z)$ is on neither the line $U + \eps V = 0$ nor the line $U - \eps V = 0$. Since there are $l$ points on each of the lines, and the intersection of the lines is $(0,0)$, there are $l^2-2l+1=(l-1)^2$ points $(U,V)$ in $\A^2(\Z/l\Z)$ such that $l\nmid I$.

Next we want to count the number of points in the subset $\set{(U,V)\in\A^2(\Z/l\Z)|l\nmid p,\;l\nmid I}$. By Lemma \ref{lem:residue} and Equation (\ref{eq:factorization-of-p}) , $p$ factors as
$$\frac{b^2}{64}\left(U-\sqrt{\frac{-2\eps\sqrt{d}}{b}}\right)\left(U+\sqrt{\frac{-2\eps\sqrt{d}}{b}}\right)\left(V-\sqrt{\frac{2\eps'\sqrt{d}}{b}}\right)\left(V+\sqrt{\frac{2\eps'\sqrt{d}}{b}}\right).$$
Thus $p$ is the product of four linear factors. Note that the zeros of $\left(U-\sqrt{\frac{-2\eps\sqrt{d}}{b}}\right)$ and the zeros of $\left(U+\sqrt{\frac{-2\eps\sqrt{d}}{b}}\right)$ are disjoint. Similarly, the zeros of $\left(V-\sqrt{\frac{2\eps'\sqrt{d}}{b}}\right)$ and the zeros of $\left(V+\sqrt{\frac{2\eps'\sqrt{d}}{b}}\right)$ are also disjoint. Also note that the intersection of the zeros of $\left(U\pm\sqrt{\frac{-2\eps\sqrt{d}}{b}}\right)$ and the zeros of $\left(V\pm\sqrt{\frac{2\eps'\sqrt{d}}{b}}\right)$ lies in the zeros of $I$. The following figure depicts the relation among these lines in $\A^2(\Z/l\Z)$.

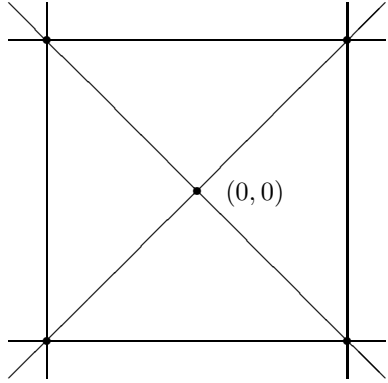
\begin{center}
\begin{picture}(10,10)
\put(0,1){\line(1,0){10}}
\put(0,9){\line(1,0){10}}
\put(1,0){\line(0,1){10}}
\put(9,0){\line(0,1){10}}
\put(0,0){\line(1,1){10}}
\put(0,10){\line(1,-1){10}}
\put(5,5){\circle*{0.2}}
\put(5.75,4.75){$(0,0)$}
\put(1,9){\circle*{0.2}}
\put(1,1){\circle*{0.2}}
\put(9,1){\circle*{0.2}}
\put(9,9){\circle*{0.2}}
\end{picture}
\end{center}

Since each line contains $l$ points in $\A^2(\Z/l\Z)$, and there are a total of six lines with five points of intersection (4 with multiplicity 3 and 1 with multiplicity 2), we conclude that the subset $\set{(U,V)\in\A^2(\Z/l\Z)|l\nmid p,\;l\nmid I}$ contains $l^2-6l+9=(l-3)^2$ points. Then the probability that $l$ does not divide either $p$ or $I$ is $\left(1-\frac{3}{l}\right)^2$.

If $l\not\in\mathcal P_C$, then by Lemma \ref{lem:divisor-of-p}, $l$ cannot be a divisor of $p$ and the probability that $l$ does not divide either $p$ or $I$ is the same as the probability that $l$ does not divide $I$, which is $\left(1-\frac{1}{l}\right)^2$ if $l\in\mathcal P_S$, and $\left(1-\frac{1}{l^2}\right)$ if $l\in\mathcal P_I$.

If $l\mid d$, note that $p\equiv I^2\mod{l}$. Hence $l$ divides $p$ if and only if $l$ divides $I$. Hence the probability of $l$ dividing neither $p$ or $I$ is the same as the probability of $l$ not dividing $I$, which is $(1-1/l)$.

For the case $l=2$, since $X$ and $Y$ are odd integers, we may write $X=2X_0+1$ and $Y=2Y_0+1$.
It is then not hard to see that both $p$ and $I$ are odd integers. Hence we have $2\nmid I$ and $2\nmid p$, which completes the proof.
\end{proof}

We use the above Proposition to determine the correction factor for $l\nmid b$. For odd primes $l\mid b$, we use the following result.

\begin{prop}
Let $l$ be an odd prime divisor of $b$. Then
\begin{equation}
\Prob(l\nmid p\textrm{ and }l\nmid I) = \left\{
                                          \begin{array}{cl}
                                            (1-3/l)^2, & \hbox{if $l\equiv 1\mod{4}$;} \\
                                            (1-1/l)^2, & \hbox{if $l\equiv 3\mod{4}$.}
                                          \end{array}
                                        \right.
\end{equation}
\end{prop}
\begin{proof}
Note that if $l$ is a divisor of $b$, then
\begin{eqnarray*}
I &\equiv& c\abs{(C+D)(C-D)} \mod{l},\\
p &\equiv& c^2(C^2+1)(D^2+1) \mod{l}.
\end{eqnarray*}

If $l\equiv 1\mod{4}$, then $-1$ is a residue modulo $l$. Let $\sqrt{-1}$ be one of the square root of $-1$ in $\Z/l\Z$. Therefore $l\mid p$ if and only if $C\equiv \pm\sqrt{-1}$ or $D\equiv \pm\sqrt{-1}$. The set $\set{(C,D)\in\A^2(\Z/l\Z)|C\equiv\pm\sqrt{-1}\textrm{ or }D\equiv\pm\sqrt{-1}}$ consists of $4l-4$ points. The set $\set{(C,D)|I\equiv 0} = \set{(C,D)|C+D\equiv 0\textrm{ or }C-D\equiv 0}$ consists of $2l-1$ points, with an intersection with the above set consisting of 4 points. Hence there are $l^2-6l+9 = (l-3)^2$ points in $\A^2(\Z/l\Z)$ such that $l\nmid I$ and $l\nmid p$.

If $l\equiv 3\mod{4}$, then $-1$ is a non-residue modulo $l$. Hence $p$ is not divisible by $l$. Therefore there are $l^2-2l+1 = (l-1)^2$ points in $\A^2(\Z/l\Z)$ such that $l\nmid I$ and $l\nmid p$.
\end{proof}

The following table summarizes the correction factors according to the prime number $l$.

\begin{center}
\begin{tabular}{|l|l|c|}
   \hline
     & prime number $l$ & correction factor $c(l)$ \\
   \hline
   \hline
   $l\nmid b$ & $l$ totally splits in $K/\Q$ & $(1-2/(l-1))^2$ \\
   $l$ odd       & $l$ splits into two primes in $K/\Q$  & 1 \\
              & $l$ is inert in $K/\Q$ & $(1+2/(l-1))$ \\
              & $l$ ramifies  & $(1+1/(l-1))$ \\
   \hline
   $l\mid b$  & $l\equiv 1\mod{4}$ & $(1-2/(l-1))^2$ \\
   $l$ odd    & $l\equiv 3\mod{4}$ & 1 \\
   \hline
   $l=2$      &  & 4 \\
   \hline
 \end{tabular}
\end{center}

Simply put, the probability that both $p$ and $I$ are prime numbers is expected to be equal to $1/\log(p)\log(I)$, which is the probability if both $p$ and $I$ are viewed as random numbers, multiplied by the global correction factor, which is the product of all correction factors corresponding to each prime $l$. We need to note the following lemma that ensures the convergence of the correction factor as an infinite product.

\begin{lem}\label{lem:convergence}
The infinite product
\begin{equation}
\lim_{B\to\infty}\prod_{l\leq B}c(l)
\end{equation}
converges conditionally if $c(3)\neq 0$, and it diverges to zero if $c(3)=0$.
\end{lem}
\begin{proof}
First note that the only correction factor that could be zero is $c(3)$, which is zero if and only if $3\in\mathcal P_C$. Now we assume that $3\not\in\mathcal P_C$ and show that the infinite product converges conditionally. Since the primes dividing $b$ form a finite set, and those primes in $\mathcal P_S$ have no contributions to the correction factor, it suffices to show that the infinite product
\begin{equation*}
\lim_{B\to\infty}\prod_{l\leq B,l\in\mathcal P_C}c(l)\prod_{l\leq B,l\in\mathcal P_I}c(l)
= \lim_{B\to\infty}\prod_{l\leq B,l\in\mathcal P_C}\left(1-\frac{2}{l-1}\right)^2\prod_{l\leq B,l\in\mathcal P_I}\left(1+\frac{2}{l-1}\right)
\end{equation*}
converges conditionally.

First observe that
\begin{eqnarray*}
& & \prod_{l\leq B,l\in\mathcal P_C} \left(1-\frac{2}{l-1}\right)^2 \prod_{l\leq B,l\in\mathcal P_I}\left(1+\frac{2}{l-1}\right) \\
&=& \prod_{l\leq B,l\in\mathcal P_C} \left(1-\frac{2}{l}\right)^2 \prod_{l\leq B,l\in\mathcal P_I}\left(1+\frac{2}{l}\right) \\
& & \cdot\left[\prod_{l\leq B,l\in\mathcal P_C}\left(1-\frac{2}{(l-1)(l-2)}\right)^2\prod_{l\leq B,l\in\mathcal P_I}\left(1+\frac{2}{(l+2)(l-1)}\right)\right].
\end{eqnarray*}
As $\prod_{l\leq B,l\in\mathcal P_C}\left(1-\frac{2}{(l-1)(l-2)}\right)^2\prod_{l\leq B,l\in\mathcal P_I}\left(1+\frac{2}{(l+2)(l-1)}\right)$ converges absolutely as $B\to\infty$, it then suffices to show that $\prod_{l\leq B,l\in\mathcal P_C} \left(1-\frac{2}{l}\right)^2 \prod_{l\leq B,l\in\mathcal P_I}\left(1+\frac{2}{l}\right)$
converges conditionally. Also note that
\begin{eqnarray*}
& & \prod_{l\leq B,l\in\mathcal P_C} \left(1-\frac{2}{l}\right)^2 \prod_{l\leq B,l\in\mathcal P_I}\left(1+\frac{2}{l}\right) \\
&=& \prod_{l\leq B,l\in\mathcal P_C} \left(1-\frac{1}{l}\right)^4 \prod_{l\leq B,l\in\mathcal P_I}\left(1+\frac{1}{l}\right)^2 \cdot \left[\prod_{l\leq B,l\in\mathcal P_C}\left(1-\frac{1}{(l-1)^2}\right)^2\prod_{l\leq B,l\in\mathcal P_I}\left(1-\frac{1}{(l+1)^2}\right)\right].
\end{eqnarray*}
Since $\prod_{l\leq B,l\in\mathcal P_C}\left(1-\frac{1}{(l-1)^2}\right)^2\prod_{l\leq B,l\in\mathcal P_I}\left(1-\frac{1}{(l+1)^2}\right)$ converges absolutely as $B\to\infty$, it then suffices to show that $ \prod_{l\leq B,l\in\mathcal P_C} \left(1-\frac{1}{l}\right)^2 \prod_{l\leq B,l\in\mathcal P_I}\left(1+\frac{1}{l}\right)$ converges conditionally. Note that
\begin{eqnarray*}
& & \prod_{l\leq B,l\in\mathcal P_C} \left(1-\frac{1}{l}\right)^2 \prod_{l\leq B,l\in\mathcal P_I}\left(1+\frac{1}{l}\right) \\
&=& \prod_{l\leq B,l\in\mathcal P_C}\left(1-\frac{1}{l}\right)\prod_{l\leq B,l\in\mathcal P_S}\left(1-\frac{1}{l}\right)^{-1}\cdot \prod_{l\leq z}\left(1-\frac{\chi(l)}{l}\right),
\end{eqnarray*}
where $\chi$ is the unique Dirichlet character modulo $l$ of order two. Since $\prod_{l\leq B}\left(1-\frac{\chi(l)}{l}\right)$ converges as $z\to\infty$, it then suffices to show that $\prod_{l\leq B,l\in\mathcal P_C}\left(1-\frac{1}{l}\right)\prod_{l\leq B,l\in\mathcal P_S}\left(1-\frac{1}{l}\right)^{-1}$ converges conditionally. Observe that
\begin{eqnarray*}
& & \prod_{l\leq B,l\in\mathcal P_C}\left(1-\frac{1}{l}\right)\prod_{l\leq B,l\in\mathcal P_S}\left(1-\frac{1}{l}\right)^{-1} \\
&=& \prod_{l\leq B,l\in\mathcal P_C}\left(1-\frac{1}{l}\right)\prod_{l\leq B,l\in\mathcal P_S}\left(1+\frac{1}{l}\right)\cdot\prod_{l\leq B,l\in\mathcal P_S}\left(1-\frac{1}{l^2}\right)^{-1}.
\end{eqnarray*}
Since $\prod_{l\leq B,l\in\mathcal P_S}\left(1-\frac{1}{l^2}\right)$ converges absolutely as $z\to\infty$, and $\prod_{l\leq B,l\in\mathcal P_C}\left(1-\frac{1}{l}\right)\prod_{l\leq B,l\in\mathcal P_S}\left(1+\frac{1}{l}\right)$ converges conditionally by the Chebotarev density theorem, we conclude that the given infinite product converges conditionally.
\end{proof}

In the above discussion, we did not take into account the case where $I$ divides $p$. However, computations show that $I$ divides $p$ with a very low probability when compared with the estimate error in the estimation by the prime number theorem.

\

\section{Examples}

\subsection{$K=\Q(\zeta_5)$}
First we apply the theory in the previous section to Weil $p$-numbers in the field $K=\Q(\zeta_5)$. Recall that the curve $y^2=x^5+1$ has CM by $K=\Q(\zeta_5)$. The ring of integers of $K$ is $\OK=\Z[\zeta_5]$. The prime $p$ and the index $I$ of a Weil $p$-number in $K$ satisfy the following equations
\begin{eqnarray}
\label{eq:1} p &=& \frac{1}{16}\left((X^2+XY-Y^2)^2 + 5X^2 + 5Y^2 + 5\right), \quad\textrm{and} \\
\label{eq:2} I &=& \frac{1}{4}\abs{X^2-4XY-Y^2},
\end{eqnarray}
where $X$ and $Y$ range among odd integers. Note that in order that both $p$ and $I$ are integers, $X$ and $Y$ must both be odd numbers. The corresponding Weil $p$-number is then
\begin{equation*}
\pi = \frac{1}{4}\left(-(X^2+XY-Y^2) + \sqrt{5} + X\sqrt{-5-2\sqrt{5}} + Y\sqrt{-5+2\sqrt{5}}\right).
\end{equation*}
It is worth noting that $\pi$ is in fact contained in $\OK$. First we note that $\sqrt{-5\pm2\sqrt{5}}\in\Z[\zeta_5]$. This is because $\zeta_5 = \cos(2\pi/5) + i\sin(2\pi/5)$, which yields that $i\sin(2\pi/5) = \frac{\sqrt{-10-2\sqrt{5}}}{4}\in\Q(\zeta_5)$. Hence $\sqrt{-10\pm2\sqrt{5}}\in\Q(\zeta_5)$. Note that $$\sqrt{-5\pm2\sqrt{5}} = \frac{1}{2}\left(\sqrt{-10-2\sqrt{5}}\pm \sqrt{-10+2\sqrt{5}}\right),$$ hence $\sqrt{-5\pm2\sqrt{5}}\in\Q(\zeta_5)$. Moreover, as $\sqrt{-5\pm2\sqrt{5}}$ satisfies the monic polynomial $X^2+10X^2+5$, we deduce that $\sqrt{-5\pm2\sqrt{5}}\in\Z[\zeta_5]$. Obviously, $4\pi\in\Z[\zeta_5]$. Thus it suffices to show that $4\pi$ is divisible by 4 in $\Z[\zeta_5]$. Note that the norm of $4\pi$ is $\left((X^2+XY-Y^2)^2+5+5X^2+5Y^2\right)^2$, we then need to show that 256 divides the norm of $4\pi$, which follows from the fact that $p$ is an integer.

Therefore, for this field $K=\Q(\zeta_5)$ we have $a_0=1$, $b=2$, $c=1$, and $d=5$. The only prime number that divides $b$ is 2. The linear transformation for $K$ is simplified as following, where for this field $\eps=-\frac{\sqrt{5}+1}{2}$. The $K_0$-linear transformation from $K_0[X,Y]\to K_0[U,V]$
\begin{eqnarray}
\label{transform:1} U &=& X + \frac{\sqrt{5}+1}{2} Y \\
\label{transform:2} V &=& X + \frac{1-\sqrt{5}}{2} Y
\end{eqnarray}
is well-defined in $\Z/l\Z$, for odd prime numbers $l\equiv 1,4\mod{5}$. Note that the determinant of this linear transformation is $-\sqrt{5}\neq 0$ as $l\neq 5$. Under this transformation, we have
\begin{eqnarray}
\label{eq:3} p &=& \frac{1}{16}\left(U^2 + \frac{5+\sqrt{5}}{2}\right)\left(V^2+\frac{5-\sqrt{5}}{2}\right), \\
\label{eq:4} I &=& \frac{1}{8}\abs{\frac{1-\sqrt{5}}{2}\,U^2 + \frac{\sqrt{5}+1}{2}\,V^2}.
\end{eqnarray}

For the field $K=\Q(\zeta_5)$, 5 is the only prime that ramifies. A prime $l$ completely splits in $K/\Q$ if and only if $l\equiv 1\mod{5}$. $l$ splits in $K_0/\Q$ if and only if $l\equiv 1,4\mod{5}$, and $l$ is inert in $K/\Q$ if and only if $l\equiv 2,3\mod{5}$. Moreover, there is no odd prime divisor of $b$. Hence we have the following table of correction factors.

\begin{center}
\begin{tabular}{|l|l|c|}
  \hline
   & Primes $l$ & Correction factors $c(l)$ \\
  \hline
  \hline
  $l$ odd & $l\equiv 1\mod{5}$ & $(1-2/(l-1))^2$ \\
   & $l\equiv 4\mod{5}$ & 1 \\
   & $l\equiv 2,3\mod{5}$ & $(1+2/(l-1))^2$ \\
   & $l=5$ & $5/4$ \\
   \hline
  $l=2$ &  & $4$ \\
  \hline
\end{tabular}
\end{center}

The above correction factors are supported by computational results, as shown in the following tables, where the frequency are counted as $X$ and $Y$ range through odd numbers between 3 and 2001. Here for each prime number $l$, the predicted frequency is $c(l)(1-\frac{1}{l})$

For prime numbers $l\equiv 1\mod{5}$, we have the following table.

\begin{center}
\begin{tabular}{|r|r|r|}
  \hline
  Prime $l$ & Actual Frequency & Predicted Frequency \\
  \hline
  \hline
  11 & 0.529075 & 0.528925620... \\
  \hline
  31 & 0.815797 & 0.815816857...\\
  \hline
  41 & 0.859037 & 0.859012493... \\
  \hline
  61 & 0.904074 & 0.904058049... \\
  \hline
  71 & 0.917314 & 0.917278318... \\
  \hline
  101 & 0.941494 & 0.941476326... \\
  \hline
  131 & 0.954744 & 0.954722918... \\
  \hline
\end{tabular}
\end{center}

For prime numbers $l\equiv 4\mod{5}$, we have the following table.

\begin{center}
\begin{tabular}{|r|r|r|}
  \hline
  Prime $l$ & Actual Frequency & Predicted Frequency \\
  \hline
  \hline
  19 & 0.897545 & 0.897506925... \\
  \hline
  29 & 0.932259 & 0.932223543... \\
  \hline
  59 & 0.966391 & 0.966388969... \\
  \hline
  79 & 0.974854 & 0.974843775... \\
  \hline
  89 & 0.977649 & 0.977654336... \\
  \hline
  109 & 0.981733 & 0.981735544... \\
  \hline
  139 & 0.985659 & 0.985663268... \\
  \hline
\end{tabular}
\end{center}

For prime numbers $l\equiv 2,3\mod{5}$, we have the following table.

\begin{center}
\begin{tabular}{|r|r|r|}
  \hline
  Prime $l$ & Actual Frequency & Predicted Frequency \\
  \hline
  \hline
  3 & 0.888444 & 0.888888889... \\
  \hline
  7 & 0.979551 & 0.979591837... \\
  \hline
  13 & 0.994071 & 0.994082840... \\
  \hline
  17 & 0.996519 & 0.996539792... \\
  \hline
  23 & 0.998064 & 0.998109641... \\
  \hline
  37 & 0.999271 & 0.999269540... \\
  \hline
  43 & 0.999471 & 0.999459167... \\
  \hline
  47 & 0.999559 & 0.999547306... \\
  \hline
  53 & 0.999639 & 0.999644001... \\
  \hline
\end{tabular}
\end{center}

For the primes $l=2,5$, the estimates is also confirmed by computation.

The probability that both $p$ and $I$ are prime numbers is predicted by the following equation.
\begin{equation}
\Prob(p\textrm{ and }I\textrm{ primes})=\frac{5}{\log p\log I}\lim_{z\to\infty}\prod_{l\equiv 1(5), l\leq z} \left(1-\frac{2}{l-1}\right)^2\prod_{l\equiv 2,3(5), l\leq z}\left(1+\frac{2}{l-1}\right).
\end{equation}

The convergence is ensured by Lemma \ref{lem:convergence}.

Computation shows that the constant $C = \prod_{l\equiv 1(5)} \left(1-\frac{2}{l-1}\right)^2\prod_{l\equiv 2,3(5)}\left(1+\frac{2}{l-1}\right)\approx 2.292...$. The convergence of this infinite product is shown in the following table. Here we write $C(z) = \prod_{l\equiv 1(5),l\leq z} \left(1-\frac{2}{l-1}\right)^2\prod_{l\equiv 2,3(5),l\leq z}\left(1+\frac{2}{l-1}\right)$.

\begin{center}
\begin{tabular}{|r|c|}
  \hline
  $z$ & $C(z)$ \\
  \hline
  \hline
  100 & 2.24789155326159
 \\
 \hline
  1000 & 2.28832917493766 \\
  \hline
  10000 & 2.28500490081341
 \\
 \hline
  100000 & 2.29169100450671
 \\
 \hline
  1000000 & 2.29206360346098 \\
  \hline
\end{tabular}
\end{center}

The following table summarizes the computational results comparing the predictions with the actual numbers of pairs $(p,I)$ such that both $p$ and $I$ are prime numbers. In the table both $X$ and $Y$ range between $1$ and the bound. The discrepancy is computed as of Prediction/Actual Number $-1$.

\begin{center}
\begin{tabular}{|r|r|r|r|}
  \hline
  Bound & Actual Number & Predicted Number & Discrepancy \\
  \hline
  \hline
  200 & 896 & 918 & 0.02455 \\
  \hline
  400 & 2575 & 2638 & 0.02447 \\
  \hline
  600 & 4833 & 5002 & 0.03497 \\
  \hline
  800 & 7759 & 7940 & 0.02332 \\
  \hline
  1000 & 11316 & 11413 & 0.00857 \\
  \hline
  1200 & 15308 & 15390 & 0.00536 \\
  \hline
\end{tabular}
\end{center}

\begin{rem}
We see from the above table that when the bound is relatively small the discrepancy gets larger. One possible reason might be the slow convergence of $C(z)$. Note that when $z$ is around 100, the value for $C(z)$ is about $2\%$ lower than $\lim_{z\to\infty}C(z)$. If we take the slow convergence of $C(z)$ into consideration and redo the computation with the correction factor $\prod_{l\leq I}c(l)$, we get somewhat better agreement, as shown in the table below.

\begin{center}
\begin{tabular}{|r|r|r|r|}
  \hline
  Bound & Actual Number & Predicted Number & Discrepancy \\
  \hline
  \hline
  200 & 896 & 908 & 0.01339 \\
  \hline
  400 & 2575 & 2624 & 0.01902 \\
  \hline
  600 & 4833 & 4980 & 0.03041 \\
  \hline
  800 & 7759 & 7909 & 0.01933 \\
  \hline
  1000 & 11316 & 11370 & 0.00477 \\
  \hline
  1200 & 15308 & 15335 & 0.00176 \\
  \hline
\end{tabular}
\end{center}

For each prime number $l$, the probability that $l$ does not divide a given number $x\geq l$ is $(1-\frac{1}{l})$. It follows from this intuition that the probability that an integer $x$ is a prime number is
\begin{equation}\label{eq:pnt-estimate-1}
\prod_{l\leq x}\left(1-\frac{1}{l}\right),
\end{equation}
where the product is over all prime numbers $l\leq x$. However, this is not correct. Computations show that the product is less than the probability given by the prime number theorem, which is $1/\log(x)$. The reason is that the divisibilities of $x$ by distinct prime numbers $l<x$ are not in fact independent. We may consider the following example. Consider $x=1000$, then there are respectively $\lfloor\frac{1000}{7}\rfloor=142$, $\lfloor\frac{1000}{11}\rfloor=99$, and $\lfloor\frac{1000}{13}\rfloor=76$ integers less than 1000 divisible by $7$,$11$, and $13$. However the first number that is divisible by all the three primes is 1001. These situations give the discrepancy in the estimates by Equation (\ref{eq:pnt-estimate-1}).

Merten's formulas suggest that
\begin{equation}\label{eq:pnt-estimate-merten}
\frac{1}{\log(x)} \sim \prod_{l<x^{e^{-\gamma}}}\left(1-\frac{1}{l}\right),
\end{equation}
where $\gamma\approx 0.57721$ is the Euler constant. In the estimates above, we would like to multiply the correction factor up to the bound $I^{e^{-\gamma}}$ for correction factors involving $I$ and up to the bound $p^{e^{-\gamma}}$ for correction factors involving $p$. Observe that for almost all the cases we have $p>I$, hence we may rewrite the correction factor as
\begin{equation}
c(l;p,I) = \prod_{2\leq l<I^{e^{-\gamma}}}c(l)\prod_{I^{e^{-\gamma}}<l<p^{e^{-\gamma}}}c_p(l),
\end{equation}
where
\begin{equation}
c(l) = \frac{\Prob(l\nmid I\textrm{ and }l\nmid p)}{(1-1/l)^2}
\end{equation}
is the correction factor as above, and
\begin{equation}
c_p(l) = \frac{\Prob(l\nmid p)}{(1-1/l)}
\end{equation}
is the correction factor for $l$ not dividing $p$ only.
\end{rem}

We note the following result on the probability $\Prob(l\nmid p)$.

\begin{prop}
Let $l$ be an odd prime number other than 5. Then
\begin{equation}
\Prob(l\nmid p) = \left\{
                    \begin{array}{ll}
                      (1-2/l)^2, & \hbox{if $l\equiv 1\mod{5}$;} \\
                      1, & \hbox{otherwise.}
                    \end{array}
                  \right.
\end{equation}
\end{prop}
\begin{proof}
As we have seen in the proof of Proposition \ref{thm:2:gen}, if $l\equiv 1\mod{5}$, then $p$ as a polynomial in $X$ and $Y$ factors into four linear factors. Note that the zeros of each polynomial factor represents a line in $\A^2(\Z/l\Z)$, and we have the figure below showing the zeros of $p$.
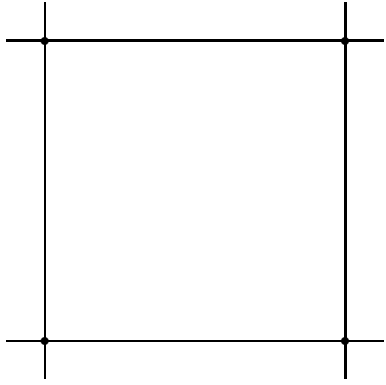
\begin{center}
\begin{picture}(10,10)
\put(0,1){\line(1,0){10}}
\put(0,9){\line(1,0){10}}
\put(1,0){\line(0,1){10}}
\put(9,0){\line(0,1){10}}
\put(1,9){\circle*{0.2}}
\put(1,1){\circle*{0.2}}
\put(9,1){\circle*{0.2}}
\put(9,9){\circle*{0.2}}
\end{picture}
\end{center}

Note that each line consists of $l$ points. There are four intersections, each with multiplicity two. Therefore the zero of $p$ consists of $4l-4$ points, which shows that the probability of $l$ not dividing $p$ is $\frac{l^2-4l+4}{l^2}=(1-2/l)^2$.

If $l\equiv 2,3,4\mod{5}$, then by Lemma \ref{lem:divisor-of-p}, $l$ does not divide $p$. Hence the probability of $l$ not dividing $p$ is 1. Thus we finish the proof.
\end{proof}

We shall also note that

\begin{prop} The infinite product
\begin{equation}\label{eq:infinite-prod-p}
\prod_{5<l\leq B} c_p(l)
\end{equation}
converges conditionally as $B\to\infty$.
\end{prop}
\begin{proof}
Note that by the above Proposition we have
\begin{equation*}
\prod_{5<l\leq B} c_p(l) = \prod_{5<l\leq B,l\equiv 1\pmod{5}}\left(1-\frac{2}{l}\right)^2  \prod_{5<l\leq B} \left(1-\frac{1}{l}\right)^{-1}.
\end{equation*}
Hence
\begin{eqnarray*}
\log \prod_{5<l\leq B} c_p(l) &=& \sum_{5<l\leq B, l\equiv 1\pmod{5}} 2\log\left(1-\frac{2}{l}\right) - \sum_{5<l\leq B} \log\left(1-\frac{1}{l}\right) \\
&=& \sum_{5<l\leq B, l\equiv 1\pmod{5}} 2\left(-\frac{2}{l}+\frac{4}{l^2}+\cdots\right) - \sum_{5<l\leq B} \left(-\frac{1}{l}+\frac{1}{l^2}+\cdots\right) \\
&=& \sum_{5<l\leq B, l\equiv 1\pmod{5}} -\frac{4}{l} + \sum_{5<l\leq B} \frac{1}{l} + \textrm{an absolutely convergent series}.
\end{eqnarray*}
Further note that
\begin{equation*}
\sum_{5<l\leq B, l\equiv 1\mod{5}} -\frac{4}{l} + \sum_{5<l\leq B} \frac{1}{l}
\end{equation*}
converges conditionally by Dirichlet's Density Theorem, which shows that the infinite product (\ref{eq:infinite-prod-p}) converges conditionally.
\end{proof}

\subsection{The field $K=\Q(\sqrt{-29-2\sqrt{29}})$}

We now consider the field $K=\Q(\sqrt{-29-2\sqrt{29}})$ as an example. Note that the real subfield $K_0=\Q(\sqrt{29})$, and $\disc(K)=29^3$. Hence the only prime number that ramifies in $K$ is $29$, which ramifies totally. By the Kronecker-Weber Theorem, $K$ is the only quartic subfield of $\Q(\zeta_{29})$. Hence the factorization of primes in $K/\Q$ depends only on their residue classes modulo 29. Here is a table showing the splitting of prime numbers according to their residue classes.

\begin{center}
\begin{tabular}{|l|l|}
  \hline
  Factorization & Residue classes modulo 29 \\
  \hline
  \hline
  Totally Split in $K/Q$ & 1,7,16,20,23,24,25 \\
  \hline
  Split in $K_0/Q$ & 4,5,6,9,13,22,28 \\
  \hline
  Inert in $K/Q$ & 2,3,8,10,11,12,14,15,17,18,19,21,26,27 \\
  \hline
  Ramifies in $K/Q$ & 0 \\
  \hline
\end{tabular}
\end{center}

The correction factor for this field is

\begin{equation}
4\left(1+\frac{1}{29-1}\right)\prod_{l\in\mathcal P_C}\left(1-\frac{2}{l-1}\right)^2\prod_{l\in\mathcal P_I}\left(1+\frac{2}{l-1}\right)\approx 5.191
\end{equation}

The following table shows the discrepancies of the actual and the predicted numbers.

\begin{center}
\begin{tabular}{|r|r|r|r|}
  \hline
  Bound & Actual Number & Predicted Number & Discrepancy \\
  \hline
  \hline
  200 & 337 & 330 & $-0.02121$ \\
  \hline
  400 & 1028 & 987 & $-0.04154$ \\
  \hline
  600 & 1931 & 1904 & $-0.01418$ \\
  \hline
  800 & 3107 & 3054 & $-0.01735$ \\
  \hline
  1000 & 4491 & 4421 & $-0.01583$ \\
  \hline
  1200 & 6152 & 5995 & $-0.02618$ \\
  \hline
\end{tabular}
\end{center}

\subsection{The field $K=\Q(\sqrt{-37-6\sqrt{37}})$}

We now consider the field $K=\Q(\sqrt{-37-6\sqrt{37}})$ as an example. Note that the real subfield $K_0=\Q(\sqrt{37})$, and $\disc(K)=37^3$. Hence the only prime number that ramifies in $K$ is $37$, which ramifies totally. By Kronecker-Weber Theorem, $K$ is the only quartic subfield of $\Q(\zeta_{37})$. Hence the factorization of primes in $K/\Q$ depends only on their residue classes modulo 37. Here is a table showing the splitting of prime numbers according to their residue classes.

\begin{center}
\begin{tabular}{|l|l|}
  \hline
  Factorization & Residue classes modulo 37 \\
  \hline
  \hline
  Totally Split in $K/Q$ & 1,7,9,10,12,16,26,33,34 \\
  \hline
  Split in $K_0/Q$ & 3,4,11,21,25,27,28,30,36 \\
  \hline
  Inert in $K/Q$ & 2,5,6,8,13,14,15,17,18, \\
   & 19,20,22,23,24,29,31,32,35 \\
  \hline
  Ramifies in $K/Q$ & 0 \\
  \hline
\end{tabular}
\end{center}

For this field, $b=6$, hence the prime 3 divides $b$, and it follows that the correction factor for 3 is 1.  The correction factor for this field is

\begin{equation}
4\left(1+\frac{1}{37-1}\right)\prod_{l\in\mathcal P_C}\left(1-\frac{2}{l-1}\right)^2\prod_{l\in\mathcal P_I}\left(1+\frac{2}{l-1}\right)\approx 4.299
\end{equation}

The following table shows the discrepancies of the actual and the predicted numbers.

\begin{center}
\begin{tabular}{|r|r|r|r|}
  \hline
  Bound & Actual Number & Predicted Number & Discrepancy \\
  \hline
  \hline
  200 & 258 & 266 & 0.03101 \\
  \hline
  400 & 785 & 801 & 0.02038 \\
  \hline
  600 & 1559 & 1547 & $-0.00769$ \\
  \hline
  800 & 2457 & 2485 & 0.01140 \\
  \hline
  1000 & 3584 & 3600 & 0.00446 \\
  \hline
\end{tabular}
\end{center}

\

\section{Acknowledgement}

I owe my deepest gratitude to my supervisor, Professor Neal Koblitz, whose guidance and support from the initial to the final level enabled me to develop an understanding of the subject, whilst allowing me room to work out problems in my own way. I also  offer my thanks to Professor William Stein for his helpful discussions on the open source math software SAGE.

\

\

\bibliographystyle{amsplain}

\end{document}